\documentclass{article}
\usepackage[margin=1.2in]{geometry} 
\usepackage{amsmath,amsthm,amssymb,amsfonts}
\usepackage{xcolor}
\usepackage{mathrsfs}
\usepackage{palatino}
\usepackage{verbatim}
\usepackage{float}
\usepackage{graphicx}
\usepackage{tikz-cd}
\usepackage{marvosym}
\usepackage{hyperref}
\usepackage{ytableau}
\usepackage{multicol}
\usepackage{dynkin-diagrams}

\newcommand{\N}{\mathbb{N}}
\newcommand{\Z}{\mathbb{Z}}
\newcommand{\R}{\mathbb{R}}
\newcommand{\Q}{\mathbb{Q}}
\newcommand{\C}{\mathbb{C}}
\newcommand{\CP}{\mathbb{CP}}
\newcommand{\RP}{\mathbb{RP}}

\DeclareMathOperator{\image}{im}

\DeclareMathOperator{\OO}{O}

\DeclareMathOperator{\GL}{GL}
\DeclareMathOperator{\SO}{SO}

\DeclareMathOperator{\RRef}{Ref}

\DeclareMathOperator{\Bl}{Bl}
\DeclareMathOperator{\Fix}{Fix}

\DeclareMathOperator{\Aut}{Aut}
\DeclareMathOperator{\Mod}{Mod}
\DeclareMathOperator{\Diff}{Diff}
\DeclareMathOperator{\Homeo}{Homeo}
\DeclareMathOperator{\Isom}{Isom}

\newtheorem{thm}{Theorem}[section]
\newtheorem{cor}[thm]{Corollary}
\newtheorem{prop}[thm]{Proposition}
\newtheorem{lem}[thm]{Lemma}

\newtheorem{defn}[thm]{Definition}

\theoremstyle{remark}
\newtheorem{rmk}[thm]{Remark}

\begin{document}
\title{The Nielsen realization problem for high degree del Pezzo surfaces}
\author{Seraphina Eun Bi Lee}
\date{}
\maketitle
\begin{abstract}
Let $M$ be a smooth $4$-manifold underlying some del Pezzo surface of degree $d \geq 6$. We consider the smooth Nielsen realization problem for $M$: which finite subgroups of $\Mod(M) = \pi_0(\Homeo^+(M))$ have lifts to $\Diff^+(M) \leq \Homeo^+(M)$ under the quotient map $\pi: \Homeo^+(M) \to \Mod(M)$? We give a complete classification of such finite subgroups of $\Mod(M)$ for $d \geq 7$ and a partial answer for $d = 6$. For the cases $d \geq 8$, the quotient map $\pi$ admits a section with image contained in $\Diff^+(M)$. For the case $d = 7$, we show that all finite order elements of $\Mod(M)$ have lifts to $\Diff^+(M)$, but there are finite subgroups of $\Mod(M)$ that do not lift to $\Diff^+(M)$. We prove that the condition of whether a finite subgroup $G \leq \Mod(M)$ lifts to $\Diff^+(M)$ is equivalent to the existence of a certain equivariant connected sum realizing $G$. For the case $d = 6$, we show this equivalence for all maximal finite subgroups $G \leq \Mod(M)$.
\end{abstract}

\section{Introduction}\label{sec:introduction}
For any closed, oriented, smooth manifold $M$, consider the \emph{mapping class group} denoted $\Mod(M) := \pi_0(\Homeo^+(M))$. There is a quotient map of groups $\pi: \Homeo^+(M) \to \Mod(M)$ sending each orientation-preserving homeomorphism $f$ to its isotopy class $[f] \in \Mod(M)$. The \emph{Nielsen realization problem} asks: for which finite subgroups $G \leq \Mod(M)$ does there exist a lift $\tilde G$ of $G$ to $\Homeo^+(M)$? 

The Nielsen realization problem has many refinements: for any reasonable structure on $M$, we may require that the lift $\tilde G$ be contained in the automorphism group $\Aut(M) \leq \Homeo^+(M)$ of this structure. Three well-studied refinements are the \emph{smooth}, \emph{metric}, and \emph{complex} Nielsen realization problems. Note that affirmative answers to the complex and metric Nielsen realization problems imply an affirmative answer to the smooth Nielsen realization problem. 

For surfaces $M$, all three Nielsen realization problems were answered affirmatively for cyclic groups $G \leq \Mod(M)$ by Nielsen (\cite{nielsen}), for solvable groups $G \leq \Mod(M)$ by Fenchel (\cite{fenchel}), and for a general finite group $G \leq \Mod(M)$ by Kerckhoff (\cite{kerckhoff}). For $4$-manifolds, these three Nielsen realization problems were first studied by Farb--Looijenga (\cite{farb--looijenga}), in which they solve the metric and complex Nielsen realization problems for K3 surfaces $M$ (\cite[Theorem 1.2]{farb--looijenga}) and the smooth version for involutions (\cite[Theorem 1.8]{farb--looijenga}). Unlike the case of surfaces, some subgroups $G \leq \Mod(M)$ are realized and some are not. 

The goal of this paper is to solve the smooth Nielsen realization problem for the underlying smooth manifolds $M^4$ of del Pezzo surfaces of high degree. A \emph{del Pezzo surface} is a smooth projective algebraic surface with ample anticanonical divisor class. Any del Pezzo surface is isomorphic to $\CP^1 \times \CP^1$, $\CP^2$, or $\Bl_P\CP^2$ where $P$ is a set of $n$ points (with $1 \leq n \leq 8$) in general position (no three collinear points, no six coconic points, and no eight points on a cubic which is singular at any of the eight points); see \cite[Theorem 8.1.15, Proposition 8.1.25]{dolgachev}. The degree of the blowup $\Bl_P\CP^2$ of $\CP^2$ at $n$ points is $9-n$ and the degree of $\CP^1 \times \CP^1$ is $8$.

The blowup $\Bl_P \CP^2$ of $\CP^2$ at $n$ points is diffeomorphic to the smooth $4$-manifold 
\[
	M_n := \CP^2 \#n\overline{\CP^2}
\]
(see \cite[p. 43]{gompf--stipsicz}). Thus the underlying smooth manifolds of del Pezzo surfaces are $M_n$ with $0 \leq n \leq 8$ and $M_* = \mathbb S^2 \times \mathbb S^2$; we call these manifolds \emph{del Pezzo manifolds}. Throughout this paper, we mostly consider $M_n$ with $n = 0$, $*$, $1$, $2$, or $3$ which are the underlying manifolds of del Pezzo surfaces of degree $d \geq 6$. 

In order to study smooth actions by finite groups on $M_n$, we consider \emph{equivariant connected sums}. For some $k \geq 1$ and all $1 \leq i \leq k$, let $N_i$ be a smooth $4$-manifold with a finite group $G$ acting on $N_i$ by orientation-preserving diffeomorphisms. Under some conditions, we can $G$-equivariantly glue the manifolds $N_i$ at points $p_i \in N_i$ fixed by $G$ or along a $G$-orbit of points in $N_i$ to form a connected sum $N_1 \# \dots \# N_k$ with a smooth $G$-action such that $G$ acts on each $N_i$ in the prescribed way. See Figure \ref{fig:connected-sum} for an illustration of an equivariant connected sum. In this paper, we further impose for each $1 \leq i \leq k$ that $N_i$ or $\overline{N_i}$ be a complex surface on which $G$ acts by biholomorphisms and anti-biholomorphisms; we call such a connected sum a \emph{complex equivariant connected sum}. See Section \ref{sec:cecs} for a more precise definition and discussion.
\begin{figure}
\centering
\includegraphics[width=0.7\textwidth]{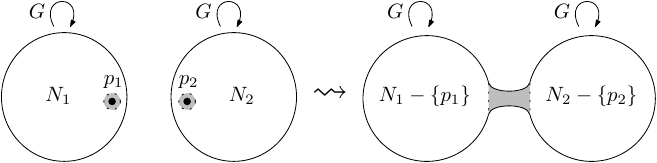}
\caption{An equivariant connected sum $(N_1 \# N_2, G)$. Left: The group $G$ acts by diffeomorphisms on both $N_1$ and $N_2$, fixes $p_1$ and $p_2$, and preserves some neighborhood $U_i$ (in grey) of each $p_i$. Right: A connected sum formed by gluing $U_i - \{p_i\}$ for $i = 1, 2$ in a $G$-equivariant way.}\label{fig:connected-sum}
\end{figure}

\bigskip
\noindent
{\bf{Main results.}} Among the del Pezzo manifolds considered in this paper, only $M_2$ and $M_3$ have infinite mapping class groups; we focus on the smooth Nielsen realization problem for $M_2$ and $M_3$ below. The reader may find the statements and proofs for the cases of $M_0$, $M_*$, and $M_1$ in Section \ref{sec:small-n}.

The following theorem gives a complete solution for the smooth Nielsen realization problem for $M_2$ in terms of the existence of complex equivariant connected sums.
\begin{thm}[\bf{Realizability Classification}]\label{thm:m2}
Let $G \leq \Mod(M_2)$ be a finite subgroup. There exists a lift $\tilde G \leq \Diff^+(M_2)$ of $G$ under $\pi: \Homeo^+(M_2) \to \Mod(M_2)$ if and only if $G$ is realized by a complex equivariant connected sum. In particular, some finite subgroups $G$ are realized by diffeomorphisms and some are not.
\end{thm}
The proof of Theorem \ref{thm:m2} uses the fact that an index $2$ subgroup of $\Mod(M_2)$ is isomorphic to a hyperbolic reflection group, which yields an enumeration of the finite subgroups of $\Mod(M_2)$ up to conjugacy. For each finite subgroup, we either construct a lift to $\Diff^+(M_2)$ by a complex equivariant connected sum or show it does not lift to $\Diff^+(M_2)$ using the theory of finite group actions on $4$-manifolds. See Section \ref{sec:ModM} for an overview of these tools. 

Some consequences of Theorem \ref{thm:m2} and its proof distinguish the Nielsen realization problem for $M_2$ from those of surfaces and K3 manifolds. For example, the proof of Theorem \ref{thm:m2} answers the smooth Nielsen realization problem for finite cyclic subgroups of $\Mod(M_2)$ affirmatively.
\begin{cor}[\bf{Smooth Nielsen realization for cyclic groups}]\label{cor:finite-order-m2}
If $c \in \Mod(M_2)$ has finite order $n$ then there exists $f \in \Diff^+(M_2)$ with order $n$ such that $[f] = c$.
\end{cor}
We record a specific case of Corollary \ref{cor:finite-order-m2} below to emphasize that the situation differs from that of K3 manifolds, in which the topological isotopy class of any Dehn twist about a $(-2)$-sphere does not lift to any finite order diffeomorphism (see Farb--Looijenga \cite[Corollary 1.10]{farb--looijenga}). 
\begin{cor}[\bf{Twists lift in $\Mod(M_2)$}]\label{cor:dehn-twist-m2}
For any Dehn twist $T$ about a $(-2)$-sphere in $M_2$, there is an order $2$ diffeomorphism $f \in \Diff^+(M_2)$ that is topologically isotopic to $T$.
\end{cor}
One way in which the smooth Nielsen realization problem for $M_2$ differs from that for surfaces in all relevant categories (smooth, metric, complex) is the nonrealizability of some finite subgroups $G \leq \Mod(M_2)$.
\begin{cor}[\bf{A subgroup that doesn't lift, but its elements do}]\label{cor:non-realizable-subgroup}
There exist finite subgroups $G \leq \Mod(M_2)$ that do not have any lift $\tilde G$ to $\Diff^+(M_2)$. In fact, there exist finite subgroups $G \leq \Mod(M_2)$ such that all elements $c \in G$ of order $n$ admit representatives $f \in \Diff^+(M_2)$ with order $n$ but such that $G$ itself does not lift to $\Diff^+(M_2)$.
\end{cor}
\begin{rmk}
The minimal subgroups $G \leq \Mod(M_2)$ of Corollary \ref{cor:non-realizable-subgroup} are isomorphic to $(\Z/2\Z)^2$. However, there exist subgroups of $\Mod(M_2)$ that are isomorphic to $(\Z/2\Z)^2$ that do lift to $\Diff^+(M_2)$.
\end{rmk}

The proof method of Theorem \ref{thm:m2} becomes unwieldy as the sizes of maximal finite subgroups of $\Mod(M_n)$ grow as $n$ grows. Instead of a full solution, we answer the smooth Nielsen realization problem only for maximal finite subgroups of $\Mod(M_3)$ in terms of complex equivariant connected sums.
\begin{thm}[\bf{Realizing maximal finite subgroups}]\label{thm:m3}
Up to conjugation, $\Mod(M_3)$ has three maximal finite subgroups. Two of these have no lifts to $\Diff^+(M_3)$ under $\pi: \Homeo^+(M_3) \to \Mod(M_3)$. One does lift, and is in fact realized by a complex equivariant connected sum.
\end{thm}
\begin{rmk}
More specifically, consider $M_3$ as the complex surface $\Bl_P \CP^2$ where $P = \{[1:0:0], [0:1:0], [0:0:1]\}$. Let $\Aut(M_3) \leq \Diff^+(M_3)$ be the complex automorphism group of $M_3$, and $\tau: M_3 \to M_3$ be the anti-biholomorphism induced by complex conjugation on $\CP^2$. The proof of Theorem \ref{thm:m3} shows that there exists a lift $\tilde G \leq \Diff^+(M_3)$ of $G$ under $\pi: \Homeo^+(M_3) \to \Mod(M_3)$ if and only if $G$ is conjugate to $\pi(\langle \Aut(M_3), \tau \rangle)$ in $\Mod(M_3)$.
\end{rmk}

Finally, we consider the complex Nielsen realization problem for all $4$-manifolds of the form $M = M_*$ and $M_n$ for all $n \geq 0$. 
\begin{thm}[\bf{Smooth not complex}]\label{thm:complex-nielsen}
If $M = M_*$ or $M_n$ for $n \geq 0$ there exist mapping classes $c \in \Mod(M)$ of order $2$ such that there exist involutions $f \in \Diff^+(M)$ with $[f] = c$ (in fact, $c$ is realized by a complex equivariant connected sum) but such that
\begin{enumerate}
\item there exist no biholomorphic involution $f$ with $[f] = c$ for any complex structure of $M$, and 
\item if $M = M_n$ with $n \geq 1$ then there exist no anti-biholomorphic involution $f$ with $[f] = c$ for any complex structure of $M$.
\end{enumerate}
\end{thm}

\bigskip
\noindent
{\bf{Related work.}} Hambleton--Tanase (\cite[Theorem A]{hambleton--tanase}) shows that if $G = \Z/p\Z$ acts smoothly on $\#n \CP^2$ for $n \geq 1$ and $p$ is an odd prime then there exists an equivariant connected sum of linear actions on $\CP^2$ with the same fixed-set data (see \cite{hambleton--tanase} for the exact description of this data) and the same induced action on $H_2(\#n \CP^2; \Z)$. Their method is to analyze the equivariant Yang--Mills moduli space to produce a stratified $G$-equivariant cobordism between $(\#n \CP^2, G)$ and an equivariant connected sum of linear actions on $\CP^2$. Our results are similar in flavor in that we relate the existence of smooth actions on del Pezzo manifolds to the existence of complex equivariant connected sums. However, our methods are much more elementary than those of \cite{hambleton--tanase} and conversely yield less refined results in terms of the fixed sets.

Finite group actions by complex or symplectic automorphisms of blowups of $\CP^2$ have also been well-studied. All possible groups appearing as the complex automorphism groups of del Pezzo surfaces and their actions on their second homology groups have been determined by Dolgachev--Iskovskikh (\cite{dolgachev--iskovskikh}). Finite groups of symplectic automorphisms of blowups of $\CP^2$ have also been studied by Chen--Li--Wu (\cite{chen--li--wu}).

The existence of order-$2$ mapping classes of $4$-manifolds that do not lift to an order-$2$ diffeomorphism was known in a few cases; see Raymond--Scott (\cite[Theorem 1]{raymond--scott}) for the case of certain nil-manifolds (in every dimension $d \geq 3$) and Baraglia--Konno (\cite[Theorem 1.2]{baraglia--konno}) for the case of the K3 manifold. On the other hand, recent work of Farb--Looijenga (\cite{farb--looijenga}) more systematically addresses the metric, complex, and smooth Nielsen realization problems for the K3 manifold. Their result on the nonrealizability of Dehn twists in the K3 manifold by finite-order diffeomorphisms (\cite[Corollary 1.10]{farb--looijenga}) was later extended to all smooth spin $4$-manifolds with non-zero signature by Konno (\cite[Theorem 1.1]{konno}).

\bigskip
\noindent
{\bf{Outline of paper.}} In Section \ref{sec:ModM} we review the relevant facts about the mapping class groups of del Pezzo manifolds and their relationship to hyperbolic reflection subgroups. In Section \ref{sec:small-n} we detail the elementary cases of the smooth Nielsen realization problem for $M_0$, $M_*$, and $M_1$. In Section \ref{sec:m2-m3} we review some results on finite group actions on $4$-manifolds (Subsection \ref{sec:finite-gp}) and apply them in Subsection \ref{sec:m2} to classify all finite subgroups of $\Mod(M_2)$ which lift to $\Diff^+(M_2)$ and prove Theorem \ref{thm:m2} and its corollaries. We then apply similar techniques to obtain a partial result for finite subgroups of $\Mod(M_3)$ in Subsection \ref{sec:m3}. Finally, we address the complex Nielsen realization problem for $M = M_*$ and $M_n$ for all $n \geq 0$ in Subsection \ref{sec:complex-automorphisms}.

\bigskip
\noindent
{\bf{Acknowledgements.}}
I am grateful to Benson Farb for his continued support and guidance throughout every step of this project, from suggesting this problem to commenting on many previous drafts. I thank both Farb and Eduard Looijenga for sharing an early draft of their paper (\cite{farb--looijenga}) with me, which helped shape this project in its beginnings. I thank Hokuto Konno for bringing many relevant references on Nielsen realization for $4$-manifolds to my attention. I thank Danny Calegari and Shmuel Weinberger for their helpful answers to my questions about mapping class groups of and finite group actions on $4$-manifolds, and R. \.{I}nan\c{c} Baykur, Dan Margalit, Anubhav Mukherjee, and an anonymous referee for their comments on an earlier draft that improved the exposition of this paper.

\section{Mapping class groups of del Pezzo manifolds}\label{sec:ModM}

In this section we outline the tools used to study the finite subgroups of mapping class groups of del Pezzo manifolds.

\subsection{Mapping class groups of del Pezzo manifolds}

The mapping class groups $\Mod(M) := \pi_0(\Homeo^+(M))$ of closed, oriented, and simply connected $4$-manifolds are computable due to the following theorems of Freedman and Quinn.
\begin{thm}[Freedman \cite{freedman}, Quinn \cite{quinn}]\label{thm:freedman--quinn}
Let $M^4$ be a closed, oriented, and simply connected manifold. The map
\[
	\Phi: \Mod(M) \to \Aut(H_2(M; \Z), Q_M)
\]
given by $\Phi: [f] \mapsto f_*$ is an isomorphism of groups.
\end{thm}

The Mayer--Vietoris sequence implies that $H_2(M_n; \Z) = H_2(\CP^2; \Z) \oplus H_2(\overline{\CP^2}; \Z)^{\oplus n}$ and gives a natural $\Z$-basis $\{H, E_1, \dots, E_n\}$ with intersection form $Q_{M_n} \cong \langle 1 \rangle \oplus n \langle -1 \rangle$; the group $\Aut(H_2(M; \Z), Q_{M_n})$ is the indefinite orthogonal group $\OO(1,n)(\Z)$, i.e. by Freedman--Quinn (Theorem \ref{thm:freedman--quinn}),
\[
	\Mod(M_n)\cong \OO(1,n)(\Z).
\]
We will identify $\Aut(H_2(M; \Z), Q_M)$ and $\Mod(M)$ for all $M$ in the rest of this paper. 

On the other hand, there is a diffeomorphism $M_2 \cong (\CP^1 \times \CP^1) \# \overline{\CP^2}$. So in addition to the standard $\Z$-basis $\{H, E_1, E_2\}$ of $H_2(M_2; \Z)$, there is another natural $\Z$-basis $\{S_1, S_2, \Sigma\}$ of homology corresponding to the decomposition $H_2(M_2) \cong H_2(\CP^1 \times \CP^1; \Z) \oplus H_2(\overline{\CP}^2; \Z)$. The lattice $(H_2(\CP^1 \times \CP^1; \Z), Q_{\CP^1 \times \CP^1})$ has two isotropic generators $S_1$ and $S_2$ with $Q_{\CP^1 \times \CP^1}(S_1, S_2) = 1$ coming from the factors of the product $\CP^1 \times \CP^1$.

Combining the diffeomorphism $M_n \cong (\CP^1 \times \CP^1) \# (n-1)\overline{\CP}^2$ for $n \geq 2$ with Theorem \ref{thm:freedman--quinn} and applying \cite[Theorem 2]{wall-diffeos} to $M_n$ yields the following statement. (The same statement holds for $M_0$, $M_*$, and $M_1$ but these cases will be handled in Section \ref{sec:small-n}.)
\begin{thm}[{A rephrasing of \cite[Theorem 2]{wall-diffeos}}]\label{thm:diffeo-realizable}
For $M = M_*$ or $M_n$ with $2 \leq n \leq 9$, the restriction of $\pi: \Homeo^+(M) \to \Mod(M)$ to the subgroup $\Diff^+(M) \leq \Homeo^+(M)$ is surjective.
\end{thm}
\begin{rmk}
Theorem \ref{thm:diffeo-realizable} cannot be extended to manifolds $M_n$ for $n \geq 10$; Friedman--Morgan (\cite[Theorem 10]{friedman--morgan}) shows that the image of the quotient $\pi|_{\Diff^+(M_n)}: \Diff^+(M_n) \to \Aut(H_2(M_n; \Z), Q_{M_n})$ has infinite index in $\Aut(H_2(M_n; \Z), Q_{M_n})$ for all $n \geq 10$.
\end{rmk}

\subsection{Complex equivariant connected sums}\label{sec:cecs}
A definition of equivariant connected sums can be found in \cite[Section 1.C]{hambleton--tanase}. We include the definition here for the convenience of the reader. 

Let $N_1, N_2$ be smooth manifolds and $G$ a finite group. Suppose $G \leq \Diff^+(N_i)$ for $i = 1, 2$ and that there are points $p_i \in N_i$ for $i = 1, 2$ such that $p_i$ is fixed by all $g \in G$ and the tangential representations $G \to \SO(T_{p_i}N_i)$ are equivalent by an orientation-reversing isomorphism $\rho: T_{p_1}N_1 \to T_{p_2}N_2$. By the equivariant tubular neighborhood theorem (\cite[Theorem VI.2.2]{bredon}), there exist $G$-invariant neighborhoods of $p_i \in N_i$ for each $i = 1, 2$ which are $G$-equivariantly diffeomorphic to $T_{p_i}N_i$. We can now form as usual a connected sum $N_1 \# N_2$ by taking the $G$-equivariant neighborhoods of $p_1$ and $p_2$ in $N_1 - p_1$ and $N_2-p_2$ respectively and equivariantly identifying concentric annuli around $p_1$ and $p_2$ via the orientation-reversing map $\rho$. Then the connected sum $N_1 \# N_2$ has a natural smooth action of $G$. The $G$-manifold $(N_1 \# N_2, G)$ is called an \emph{equivariant connected sum}. 

More generally, suppose that $G \leq \Diff^+(N_1)$ is finite. Let $H \leq G$ be any subgroup acting smoothly on $N_2$, i.e. there is a homomorphism $H \to \Diff^+(N_2)$. Then the twisted product $G \times_H N_2$ is diffeomorphic to a disjoint union of $\lvert G/H\rvert$-many copies of $N_2$. Note that $G\times_HN_2 \not\cong G/H \times N_2$ as $G$-spaces; see \cite[Chapter I, Section 6(A)]{bredon} for more details. 

Suppose there exist points $p_1 \in N_1$ and $p_2 \in N_2$, both with stabilizers equal to $H$ such that the tangential representations $H \to \SO(T_{p_i}N_i)$ are equivalent by an orientation-reversing isomorphism $T_{p_1}N_1 \to T_{p_2}N_2$. By \cite[Theorem VI.2.2]{bredon} again, there exist $H$-invariant tubular neighborhoods of the points in the $G$-orbits of $p_1 \in N_1$ and $(1, p_2) \in G\times_HN_2$. The $G$-equivariant identification of these $H$-invariant neighborhoods forms a connected sum denoted by $N_1 \# (G \times_HN_2)$ with a natural smooth action of $G$. Letting $m = \lvert G/H\rvert$, note that $N_1 \# (G \times_HN_2)$ is diffeomorphic to $N_1 \# m N_2$. The $G$-manifold $(N_1\#mN_2, G)$ is also called an \emph{equivariant connected sum.} Note that the first construction $(N_1 \# N_2, G)$ is a special case of this more general construction with $H = G$. 

With these definitions in mind, we define a \emph{complex equivariant connected sum}. 
\begin{defn}\label{defn:cecs}
Let $M$ be a smooth manifold and let $G\leq \Diff^+(M)$ be finite. The pair $(M, G)$ is called a \emph{complex equivariant connected sum} if one of the following holds:
\begin{enumerate}
	\item $(M, G)$ is $G$-equivariantly diffeomorphic to $N$ or $\overline N$, where $N$ is a complex manifold and each $g \in G \leq \Diff^+(N)$ is biholomorphic or anti-biholomorphic, or
	\item $(M, G)$ is $G$-equivariantly diffeomorphic to an equivariant connected sum $(N_1 \# (G\times_H N_2), G)$, where $(N_1, G)$ and $(N_2, H)$ are complex equivariant connected sums with $H \leq G$.
\end{enumerate}
\end{defn}

If $G_0 \leq \Mod(M)$ is a finite group such that there exists a complex equivariant connected sum $(M, G)$ and $G \leq \Diff^+(M)$ is a lift of $G_0$ under the quotient $\pi: \Homeo^+(M) \to \Mod(M)$ then we say that $G_0$ is \emph{realizable by a complex equivariant connected sum}.

\subsection{The group $\OO(1,n)(\Z)$ and the hyperboloid model}\label{sec:o-1-n}
Fix $n \in \N$. Consider the vector space $\R^{n+1}$ with the diagonal binary symmetric form $Q_n$ of signature $(1,n)$
\[
	Q_n((x_0, x_1, \dots, x_n), (y_0, y_1, \dots, y_n)) = x_0 y_0 - x_1 y_1 - \dots - x_n y_n.
\] 
We denote the pair $(\R^{n+1}, Q_n)$ by $\mathbb E^{1,n}$ and the pair $(\R^{n+1}, R_n)$ by $\mathbb E^{n,1}$ where $R_n = -Q_n$. There is a natural isometric inclusion $(H_2(M_n; \Z), Q_{M_n}) \hookrightarrow \mathbb E^{1,n}$; using this embedding, identify $\R^{n+1}$ with the $\R$-span of the $\Z$-basis $\{H, E_1, \dots, E_n\}$ of $H_2(M_n; \Z)$ which makes the $\R$-bilinear extension of $Q_{M_n}$ coincide with $Q_n$. Under this identification, $Q_n(H, H) = -R_n(H,H) = 1$ and $Q_n(E_k, E_k) = -R_n(E_k, E_k) = -1$ for all $1 \leq k \leq n$. 

Let $\OO(n,1)(\R) \leq \GL(n+1)(\R)$ be the group of matrices preserving the form $R_n$. The group $\OO(n,1)(\Z)$ is the subgroup of integral matrices of $\OO(n,1)(\R)$. Every $v \in \R^{n+1}$ with $R_n(v,v) = \pm 1, \pm 2$ defines a \emph{reflection} $\RRef_v \in \OO(n,1)(\R)$ about $v$ by
\[
	\RRef_v(w) := w - \frac{2R_n(v,w)}{R_n(v,v)}v = w - \frac{2Q_n(v,w)}{Q_n(v,v)}v.
\]
If $v \in \Z^{n+1} \subseteq \mathbb E^{n,1}$ then $\RRef_v\in \OO(n,1)(\Z)$.

Consider the submanifold 
\[
	\mathbb H^n = \{v = (v_0, v_1, \dots, v_n)\in \R^{n+1} : R_n(v,v) = -1, \, v_0 > 0\};
\] 
the restriction of $R_n$ to $\mathbb H^n$ defines a Riemannian metric on $\mathbb H^n$. As the notation suggests, this Riemannian manifold is isometric to the hyperbolic $n$-space and is called the \emph{hyperboloid model} (see \cite[Chapter 2]{thurston}). Let $\OO^+(n,1)(\R)$ denote the index $2$ subgroup of $\OO(n,1)(\R)$ that preserves the submanifold $\mathbb H^n$. This is the isometry group $\Isom(\mathbb H^n)$ of $\mathbb H^n$. The subgroup $\OO^+(n,1)(\Z)$ is defined to be the subgroup of integral matrices of $\OO^+(n,1)(\R)$, so $\OO^+(n,1)(\Z)$ is a discrete subgroup of $\Isom(\mathbb H^n)$. Finally, we observe that $\OO(1,n)(\R) = \OO(n,1)(\R)$ as subgroups of $\GL(n+1)(\R)$. 

\subsection{Coxeter theory}\label{sec:coxeter}
According to Vinberg (\cite{vinberg}), the groups $\OO^+(n,1)(\Z)$ each contain a finite index, hyperbolic reflection subgroup acting by isometries on $\mathbb H^n$ with a fundamental domain of finite volume for all $n \leq 17$. It turns out that for $n \leq 9$, the maximal reflection subgroup of $\OO^+(n,1)(\Z)$ is $\OO^+(n,1)(\Z)$ itself. Explicit generators for each $\OO^+(n,1)(\Z)$ are also determined in \cite{wall-indefinite-orthogonal}.
\begin{thm}[{Wall, \cite[Theorem 1.5, 1.6]{wall-indefinite-orthogonal}}]\label{thm:reflections-2,3}
For $n = 2$ and $3$, the groups $\OO^+(n,1)(\Z)$ are:
\begin{align*}
	\OO^+(2,1)(\Z) &= \langle \RRef_{H-E_1-E_2}, \, \RRef_{E_1-E_2}, \, \RRef_{E_2} \rangle, \\
	\OO^+(3,1)(\Z) &= \langle \RRef_{H-E_1-E_2-E_3},\, \RRef_{E_1-E_2},\, \RRef_{E_2-E_3},\, \RRef_{E_3}  \rangle.
\end{align*}
\end{thm}
\begin{rmk}
Another way to phrase the first half of Theorem \ref{thm:reflections-2,3} is that $\OO^+(2,1)(\Z)$ is the triangle group $\Delta(2, 4, \infty)$. This formulation is classical, shown by Fricke in \cite[p. 64-68]{fricke}.
\end{rmk}

In particular, the groups $\OO^+(n,1)(\Z)$ for $n = 2, 3$ are Coxeter groups and their Coxeter diagrams are given in Figure \ref{fig:coxeter-diagrams}. Denote the Coxeter system given by the diagrams in Figure \ref{fig:coxeter-diagrams} by $(W(n), S(n))$ for $n = 2, 3$. 

\begin{figure}
\centering
\includegraphics{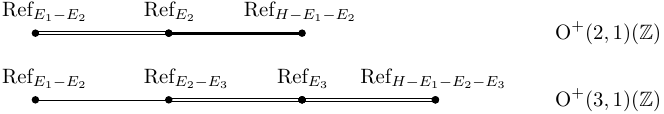}
\caption{Coxeter diagrams for $\OO^+(2,1)(\Z)$ and $\OO^+(3,1)(\Z)$.}
\label{fig:coxeter-diagrams}
\end{figure}

For the sake of completeness, we emphasize that the \emph{geometric representation} of $(W(n), S(n))$, defined using the conventions of \cite[Section 5.3]{humphreys}, coincides with the action of $\OO^+(n,1)(\Z)$ on $\mathbb E^{n,1}$. Let $V_n$ be the $\R$-span of $\{\alpha_s: s \in S(n)\}$. The standard symmetric bilinear form $B_n$ on $V_n$ defined by the Coxeter system $(W(n), S(n))$ is given on the basis $\{\alpha_s : s \in S(n)\}$ by
\[
	B_n(\alpha_s, \alpha_t) = -\cos\frac{\pi}{m(s,t)}.
\]
The action of $W(n)$ on $V_n$ preserving $B_n$ is defined on the generators $s \in S(n)$ by
\[
	s\cdot v = v - 2 B_n(\alpha_s, v) \alpha_s.
\]
Moreover, there is an isometry $F_n: (V_n, B_n) \to (\R^{n+1}, R_n)$, given on the basis elements of $V_n$ by $F_n(\alpha_{\RRef_{v}}) = R_n(v,v)^{-\frac 12}v$. One can check that $F_n(s \cdot v) = s\cdot F_n(v)$ for all $v \in V_n$ and $s \in S(n)$. From now on, we identify $\mathbb E^{n,1}$ with $(V_n, B_n)$ under the isometry $F_n$. 

The fact that $\OO^+(n,1)(\Z) \leq \Isom(\mathbb H^n)$ via the geometric representation of $(W(n), S(n))$ yields an easy classification of the finite subgroups of $\OO^+(n,1)(\Z) = \OO^+(1,n)(\Z)$ for $n = 2, 3$.
\begin{lem}\label{lem:max-finite-subgroups-Wn}
For $n = 2$ or $3$ and for any $\RRef_v \in S(n) - \{\RRef_{E_1-E_2}\}$, let 
\[
	G_v := \langle s \in S(n) - \{\RRef_v\} \rangle \leq W(n).
\]
The maximal finite subgroups of $W(n) \cong \OO^+(1,n)(\Z)$ are conjugate in $W(n)$ to some $G_v$.
\end{lem}
\begin{proof}
Let $G \leq \OO^+(n,1)(\Z) \cong W(n)$ be a finite subgroup. Then $G$ acts on $\mathbb H^n$ as a finite subgroup of isometries so it must fix at least one point in $\mathbb H^n$ (\cite[Corollary 2.5.19]{thurston}). The fundamental domain of $\OO^+(n,1)(\Z)$ in $\mathbb H^n$ ($n = 2, 3$) given by Vinberg's algorithm (\cite{vinberg}) has closure equal to the intersection 
\[
	P := \bigcap_{\RRef_v \in S(n)} \{w \in \mathbb H^n : R_n(w,v) \leq 0\},
\]
after conjugating the generators $S(n)$ by the element of $\OO^+(n,1)(\Z)$ which negates which $E_i$ and fixes $H$. 

On the other hand, let $U \subseteq V_n$ be the Tits cone of $W(n)$ and consider $-U = \{-x \in V_n : x \in U\}$. Observe that $F_n^{-1}(P)$ is contained in $-U$, and hence $F_n^{-1}(\mathbb H^n)$ is also contained in $-U$. Because $G$ fixes a point in $-U$, it also fixes a point in $U$.

For any $I \subseteq S(n)$, define
\[
	C_I = \left(\bigcap_{s \in I} \{w \in V_n : B_n(w, \alpha_s) = 0\}\right) \cap \left( \bigcap_{s \in S(n) - I} \{w \in V_n: B_n(w,\alpha_s) > 0\}\right).
\]
The family $\mathcal C$ of the sets of the form $w(C_I)$ for all $w \in W(n)$ and $I \subseteq S(n)$ partitions $U$ (\cite[Section 5.13]{humphreys}). The stabilizer of any point in $C_I$ is $W_I$ by \cite[Theorem 5.13]{humphreys}, where
\[
	W_I := \langle s \in I \subseteq S(n) \rangle \leq W(n).
\] 
If $I = S(n)$ then the only point in $V_n$ that is fixed by $W_I$ is $0 \in V_n$, which is not contained in $F_n^{-1}(\mathbb H^n)$. If $I = S(n) - \{\RRef_{E_1-E_2}\}$ then the fixed subspace of $W_I$ in $V_n$ is $F_n^{-1}(\R\{H-E_1\})$, which has empty intersection with $F_n^{-1}(\mathbb H^n)$. Note that if $I$ is a proper subset of $S(n)$ and $I \neq S(n) - \{\RRef_{E_1-E_2}\}$ then $W_I$ is contained in $G_v$ for some $\RRef_v \in S(n) - \{\RRef_{E_1-E_2}\}$. Hence if $G$ fixes a point in $F_n^{-1}(\mathbb H^n)$ then $G$ is contained, up to conjugacy in $W(n)$, in $G_v$ for some $\RRef_v \in S(n) - \{\RRef_{E_1-E_2}\}$. 
\end{proof}

For completeness, we record the analogous result for $\OO(1,n)(\Z)$ for $n = 2, 3$.
\begin{lem}\label{lem:max-finite-subgroups}
For $n = 2, 3$, maximal finite subgroups of $\OO(1,n)(\Z)$ are conjugate in $\OO(1,n)(\Z)$ to subgroups of the form $\langle G_v, -I_{n+1} \rangle$ for some $\RRef_v \in S(n) - \{\RRef_{E_1-E_2}\}$.
\end{lem}
\begin{proof}
Let $G$ be a maximal finite subgroup of $\OO(1,n)(\Z)$. Observe that $\langle -I_{n+1}, G \rangle$ is finite because $-I_{n+1} \in Z(\OO(1,n)(\Z))$, so $-I_{n+1} \in G$. Then $G$ fits into a split short exact sequence
\[
	1 \to \OO^+(1,n)(\Z) \cap G \to G \to \langle -I_{n+1} \rangle \to 1,
\]
meaning that $G = \langle\OO^+(1,n)(\Z)\cap G, -I_{n+1} \rangle$ with $\OO^+(1,n)(\Z)\cap G$ a maximal finite subgroup of $\OO^+(1,n)(\Z)$. Finally, conclude by applying Lemma \ref{lem:max-finite-subgroups-Wn}.
\end{proof}

\section{A section of $\pi: \Homeo^+(M) \to \Mod(M)$ for $M = M_0$, $M_*$ and $M_1$}\label{sec:small-n}
Let $M = \CP^2$, $\CP^1 \times \CP^1$, or $\CP^2 \# \overline{\CP^2}$. The mapping class group $\Mod(M)$ is isomorphic to $\Z/2\Z$ for $M = \CP^2$ and to $(\Z/2\Z)^2$ in the latter two cases. It turns out that there exists a particularly nice section of the quotient map $\pi: \Homeo^+(M) \to \Mod(M)$. We construct this section in the following proposition as a warmup for the rest of this paper.
\begin{prop}\label{prop:m-star-and-m1}
Let $M = \CP^2$, $\CP^1 \times \CP^1$, or $\CP^2 \# \overline{\CP^2}$. There is a section of $\pi: \Homeo^+(M) \to \Mod(M)$ with image in $\Diff^+(M)$. In fact, $\Mod(M)$ is realized by a complex equivariant connected sum.
\end{prop}

A main tool to construct complex equivariant connected sums is the following lemma. 
\begin{lem}\label{lem:glue}
Let $G_M \cong (\Z/2\Z)^2 \leq \Diff^+(M)$ and let $G_N \cong (\Z/2\Z)^2 \leq \Diff^+(N)$ and fix a group isomorphism $\Phi: G_M \to G_N$. Suppose there exist $p \in \Fix(G_M) \subseteq M$, $q \in \Fix(G_N) \subseteq N$. For all $h \in G_M$, let $F_h$ and $F_{\Phi(h)}$ denote the connected components of $p, q$ in $\Fix(h)$ and $\Fix(\Phi(h))$ respectively and suppose that $F_{h}, F_{\Phi(h)}$ are $2$-dimensional. Moreover for some pair of generators $f, g$ of $G_M$, suppose that $F_f \cap F_g$ and $F_{\Phi(f)} \cap F_{\Phi(g)}$ are $1$-dimensional.

There exist diffeomorphisms $h \# \Phi(h) \in \Diff^+(M \#\overline N)$ for all $h \in G_M$ such that $\langle h \# \Phi(h) : h \in G_M \rangle \cong G_M$ and 
\[
	[h \# \Phi(h)] = ([h], [\Phi(h)]) \in \Mod(M) \times \Mod(\overline N) \leq \Mod(M \#\overline N).
\]
\end{lem}
\begin{proof}
Fix a $G_M$-invariant metric on $M$ and $G_N$-invariant metric on $N$. The tangential representations $\rho_M: G_M \to \SO(T_pM)$ and $\rho_N: G_N \to \SO(T_qN)$ are faithful because isometries of compact manifolds are determined by their action on a point and a frame. The invariant subspaces $T_pM^{G_M}$ and $T_qN^{G_N}$ are $1$-dimensional because $F_f \cap F_g$ and $F_{\Phi(f)} \cap F_{\Phi(g)}$ are $1$-dimensional. Therefore $-I_4$ is not in the image of the tangential representations of either $G_M$ and $G_N$. 

There is a unique faithful representation $\rho: (\Z/2\Z)^2 \to \SO(4, \R)$ up to conjugation in $\SO(4, \R)$ such that $-I_4 \notin \image(\rho)$. Therefore, $\rho_M$ and $\rho_N$ are equivalent by an orientation-preserving isomorphism $T_pM \to T_qN$, i.e. an orientation-reversing isomorphism $T_pM \to T_q\overline N$. Now construct the connected sum at small $G_M$- and $G_N$-invariant disks in $M$ and $N$ centered at $p$ and $q$ respectively. Then $M \# \overline N$ becomes a smooth $(\Z/2\Z)^2$-manifold in the standard way since the map 
\[
	(f \# \Phi(f))(x)= \begin{cases}
	f(x) & x \in M - p \\
	\Phi(f)(x) & x \in N-q
	\end{cases}
\]
is a well-defined diffeomorphism of $M \# \overline N$. By construction, 
\[
	[f \# \Phi(f)] = ([f], [\Phi(f)]) \in \Mod(M) \times \Mod(\overline N) \leq \Mod(M \# \overline N). \qedhere
\]
\end{proof}

\begin{proof}[Proof of Proposition \ref{prop:m-star-and-m1}]
Consider the cases $M = M_0$, $M = M_*$ and $M = M_1$ separately.
\begin{enumerate}
	\item Let $M = M_0 = \CP^2$. Then $\Aut(H_2(M; \Z), Q_M) \cong \Z/2\Z$ with generator $c: H \mapsto -H$. The map $\tau: M \to M$ given by complex conjugation realizes $c$ so $s: c \mapsto \tau$ defines a desired section $s: \Mod(M) \to \Diff^+(M)$.
	\item Let $M = M_* = \CP^1 \times \CP^1$. Let $(S_1, S_2)$ be a $\Z$-basis of $H_2(M; \Z)$ where $S_1, S_2$ correspond to the first and second factors $\CP^1$ respectively so that $Q_{M}(S_i, S_j) = 1-\delta_{ij}$ for $1 \leq i, j \leq 2$. Then $\Aut(H_2(M; \Z), Q_{M}) = \langle c_1, c_2 \rangle \cong (\Z/2\Z)^2$ where 
	\[
		c_1 = \begin{pmatrix}-1 & 0 \\ 0 & -1\end{pmatrix},\quad\quad c_2 = \begin{pmatrix}0 & 1 \\ 1 & 0\end{pmatrix}
	\]
	with respect to the $\Z$-basis $(S_1, S_2)$. For $c = c_1$ and $c_2$, define $f_c: M \to M$ by
	\[
	f_{c_1}: ([X:Y],[W:Z]) \mapsto ([\overline X: \overline Y],[\overline W: \overline Z]), \qquad f_{c_2}: ([X:Y],[W:Z]) \mapsto ([W:Z], [X:Y]).
	\]
	Because $f_{c_1}$ and $f_{c_2}$ have order two and commute, $\langle f_{c_1}, f_{c_2} \rangle \cong (\Z/2\Z)^2 \leq \Diff^+(M)$ with $[f_{c_i}] = c_i$ for $i = 1,2$. Therefore, $s: c_i \mapsto f_{c_i}$ defines a desired section $s: \Mod(M) \to \Diff^+(M)$.

	\item Let $M = M_1 = \CP^2 \# \overline{\CP^2}$. Then $\Aut(H_2(M; \Z); \Q_M) = \langle c_1, c_2 \rangle \cong (\Z/2\Z)^2$ where
	\[
		c_1 = \begin{pmatrix}
		- 1 & 0 \\ 0 & 1
		\end{pmatrix}, \quad\quad c_2 = \begin{pmatrix}
		1 & 0 \\ 0 & - 1
		\end{pmatrix}
	\]
	with respect to the $\Z$-basis $(H, E_1)$. Define the diffeomorphisms $f_1, f_2: \CP^2 \to \CP^2$ and $g_1, g_2: \CP^2 \to \CP^2$ by
	\[
	f_1, g_2: [X:Y:Z] \mapsto [\overline X: \overline Y: \overline Z], \quad\quad f_2, g_1: [X:Y:Z] \mapsto [-X:Y:Z].
	\]
	Because $f_1$ and $f_2$ have order two and commute, $\langle f_1, f_2 \rangle \cong (\Z/2\Z)^2 \leq \Diff^+(\CP^2)$. By the same computation, $\langle g_1, g_2 \rangle \cong (\Z/2\Z)^2 \leq \Diff^+(\CP^2)$. Apply Lemma \ref{lem:glue} with $\langle f_1, f_2 \rangle \leq \Diff^+(\CP^2)$ and $\langle g_1, g_2 \rangle \leq \Diff^+(\CP^2)$. The resulting diffeomorphisms $f_i \# g_i$ of $M_1$ generate $(\Z/2\Z)^2$ and satisfy the equalities $[f_i\#g_i] = c_i$. Therefore, $s: c_i \mapsto f_i \# g_i$ defines a desired section $s: \Mod(M) \to \Diff^+(M)$. \qedhere
\end{enumerate}
\end{proof}

\section{Nielsen realization problem for del Pezzo manifolds} \label{sec:m2-m3}

\subsection{Finite group actions on $4$-manifolds}\label{sec:finite-gp}
If $G = \Z/p\Z$ with $p \in \Z$ prime acts smoothly on a closed, oriented $4$-manifold in an orientation-preserving way then its fixed set is a finite disjoint union of isolated points and surfaces (see e.g. \cite[Proof of Lemma 3.5 (3)]{farb--looijenga}). In this section we outline some results giving homological restrictions on these fixed sets. 

Any topological action of $G = \Z/p\Z$ on a closed, oriented, simply-connected $4$-manifold $M$, induces an action of $G$ on $H^2(M; \Z)$. By \cite[Proposition 1.1]{edmonds}, $H^2(M; \Z)$ decomposes as a direct sum of indecomposable representations of the trivial type, cyclotomic type, and regular type. Note that indecomposable representations of cyclotomic type are not necessarily isomorphic as $G$-representations to $\Z[\zeta_p]$, where $\zeta_p$ is a primitive $p$th root of unity, if $p \geq 23$. Similarly, not all indecomposable representations of regular type are isomorphic as $G$-representations to the group ring $\Z[G]$ if $p \geq 23$. See \cite[Proposition 1.1]{edmonds} for more details. Such exotic representations do not occur in the arguments of this paper because $p = 2$ in all applications of the following proposition. 
\begin{prop}[{Edmonds, \cite[Proposition 2.4]{edmonds}}]\label{prop:edmonds}
Let $G = \Z/p\Z$ act on a closed, oriented, simply-connected $4$-manifold $M$. Let $t$ be the number of trivial summands and $c$ be the number of cyclotomic summands in $H^2(M; \Z)$. Let $\Fix(G) \subseteq M$ be the fixed set of $G$. If $\Fix(G) \neq \emptyset$ then
\begin{enumerate}
	\item $\beta_1(\Fix(G)) = c$ and
	\item $\beta_0(\Fix(G)) + \beta_2(\Fix(G)) = t+2$,
\end{enumerate}
where $\beta_k(\Fix(G))$ denotes the mod $p$, $k$th Betti numbers of $\Fix(G)$.
\end{prop}
\begin{rmk}\label{rmk:euler-char}
By \cite[Corollary 1.4]{edmonds}, the Euler characteristic $\chi(\Fix(G))$ of $\Fix(G)$ is $t - c + 2$. Hence, if $t + 2 \neq c$ then $\Fix(G) \neq \emptyset$ if $t + 2 \neq c$ and so Proposition \ref{prop:edmonds} applies. Moreover, Proposition \ref{prop:edmonds} applies to the $\Z[G]$-module structure of $H_2(M; \Z)$ since $H_2(M; \Z)$ and $H^2(M; \Z)$ are isomorphic as $\Z[G]$-modules via Poincar\'e duality and the self-duality of the trivial, cyclotomic, and regular $G$-representations. 
\end{rmk}

The second standard result regarding fixed sets of orientation-preserving, smooth finite group actions on $4$-manifolds is the \emph{Hirzebruch $G$-signature theorem}. We specialize to the case $G = \Z/2\Z$. Two quantities are necessary in the statement of this theorem: $\sigma(M)$ and $\sigma(M/G)$. The quantity $\sigma(M)$ denotes the \emph{signature} of $M$, which is defined as $\sigma(M) = p^+ - p^-$ where $(p^+, p^-)$ is the signature of the intersection form $Q_M$ as a nondegenerate symmetric bilinear form. By \cite[Section 2.1, (22), (24)]{hirzebruch--zagier}, $\sigma(M/G)= p_G^+ - p_G^-$ where $(p_G^+, p_G^-)$ is the signature of the intersection form $Q_M$ restricted to the $G$-fixed subspace of $H_2(M; \R)$. 
\begin{thm}[{A special case of the Hirzebruch $G$-signature theorem, \cite[Section 9.2, (12)]{hirzebruch--zagier}}]\label{thm:g-signature}
Let $G = \Z/2\Z$ act on $M^4$ by orientation-preserving diffeomorphisms such that the $2$-dimensional connected components of the fixed sets are orientable. Then
\[
	2\sigma(M/G) = \sigma(M) + \sum_C Q_M([C],[C])
\]
where the sum on the right side is taken over the $2$-dimensional connected components $C$ of the fixed set of $G$ in $M$.
\end{thm}
\begin{rmk}
The Hirzebruch $G$-signature theorem in \cite[Section 9.2 (12)]{hirzebruch--zagier} is stated for a general finite group $G$. Its statement involves certain quantities called \emph{defects} which are attached to every connected submanifold of $M$ with nontrivial stabilizer in $G$. In the case that $G = \Z/2\Z$, any orientable $2$-dimensional component $C$ of the fixed set has defect equal to $Q_M([C], [C])$ according to \cite[Section 9.2, (15)]{hirzebruch--zagier}. At any isolated fixed point $p \in M$, the nontrivial element $g \in G$ acts by $-I_4$ on the tangent space $T_pM$. This observation and \cite[Section 9.2, (19)]{hirzebruch--zagier} imply that the defect at any isolated fixed point is $0$. These computations of defects reduce the general statement of the theorem to the statement given above for $G= \Z/2\Z$. 
\end{rmk}

\subsection{Smooth Nielsen realization problem for $M_2$}\label{sec:m2}
The goal of this subsection is to prove Theorem \ref{thm:m2}. We first label certain mapping classes in $\Mod(M_2)$. Denote by $\Phi, \Psi$ by following elements of $\Aut(H_2(M_2; \Z), Q_{M_2})\cong\OO(1,2)(\Z)$ whose matrix forms are given with respect to the ordered $\Z$-basis $(H, E_1, E_2)$ of $H_2(M_2; \Z)$:
\[
	\Phi = \RRef_{E_1-E_2}\RRef_{E_1} = \begin{pmatrix}
	1 & 0 & 0 \\
	0 & 0 & 1 \\
	0 & -1 & 0
	\end{pmatrix}, \quad\quad \Psi = \RRef_{E_1} =  \begin{pmatrix}
	1 & 0 & 0 \\
	0 & -1 & 0 \\
	0 & 0 & 1
	\end{pmatrix}.
\]
Denote by $A, B$ the following elements of $\Aut(H_2(M_2; \Z), Q_{M_2})\cong\OO(1,2)(\Z)$ whose matrix forms are given respect to the $\Z$-basis $(S_1, S_2, \Sigma)$:
\[
	A = \RRef_{E_1-E_2} = \begin{pmatrix}
	0 & 1 & 0 \\
	1 & 0 & 0 \\
	0 & 0 & 1
	\end{pmatrix}, \quad\quad B = \RRef_{H-E_1-E_2} = \begin{pmatrix}
	1 & 0 & 0 \\
	0 & 1 & 0 \\
	0 & 0 & -1
	\end{pmatrix}.
\]

We determine the realizable subgroups of $G_1 = \langle A, B, -I_3\rangle$ and $G_2 = \langle \Phi, \Psi, -I_3\rangle$ in Sections \ref{sec:g1} and \ref{sec:g2} respectively. In Section \ref{sec:m2}, we combine the results of Sections \ref{sec:g1} and \ref{sec:g2} to prove Theorem \ref{thm:m2} and Corollaries \ref{cor:finite-order-m2} and \ref{cor:dehn-twist-m2}.

\subsubsection{Subgroups of $G_1 = \langle A, B, -I_3 \rangle \cong (\Z/2\Z)^3$}\label{sec:g1}

In this section we first analyze the fixed sets of any lift of some $c \in G_1$ to determine some subgroups of $G_1$ which cannot be realized by diffeomorphisms. Afterwards, we explicitly realize the remaining subgroups of $G_1$. 
\begin{lem}\label{lem:fixed-A-AB}
Let $g = A$ or $-AB$. If $g$ is realized by an order $2$ diffeomorphism $f_g$ then
\[
	\Fix(f_g) \cong \mathbb S^2 \sqcup \{p\}
\]
where $p$ is an isolated fixed points of $f_g$.
\end{lem}
\begin{proof}
If $g = A$, let $S = (S_1, S_2, \Sigma)$ and if $g = -AB$, let $S = (S_1, -S_2, \Sigma)$. With respect to the $\Z$-basis $S$ of $H_2(M_2;\Z)$,
\[
	g = \begin{pmatrix}
	0 & 1 & 0 \\
	1 & 0 & 0 \\
	0 & 0 & 1
	\end{pmatrix}.
\]
By Remark \ref{rmk:euler-char}, $\chi(\Fix(f_g)) = 3$ and so $\Fix(f_g) \neq \emptyset$. Proposition \ref{prop:edmonds} implies that $\beta_1(\Fix(f_g)) = 0$ and $\beta_0(\Fix(f_g)) + \beta_2(\Fix(f_g)) = 3$. Combining with the fact that each $\Fix(\varphi)$ is a disjoint union of surfaces and isolated points shows that
\[
	\Fix(f_g) \cong \mathbb S^2 \sqcup \{p\} \text{ or }[3]
\]
where $[n]$ denotes the set of $n$ distinct isolated points. By the $G$-signature theorem (Theorem \ref{thm:g-signature}) and the fact that $\sigma(M_2) = -1$,
\[
	2\sigma(M_2/G) = -1 + \sum_C Q_{M_2}([C],[C]).
\]
Comparing the parity of both sides of the equation shows that there must exist at least one $2$-dimensional component of the fixed set. 
\end{proof}

The next lemma determines the self-intersection numbers of the submanifolds $\mathbb S^2$ of $\Fix(f_A)$ and $\Fix(f_{-AB})$.
\begin{lem}\label{lem:S-int-number}
Let $g = A$ or $-AB$. Suppose $g$ is realized by an order $2$ diffeomorphism $f_g$ or $f_{-AB}$ respectively with $\Fix(f_g) \cong F_g \sqcup \{p\}$ and $F_g \cong \mathbb S^2$. Then 
\[
	Q_{M_2}([F_A], [F_A]) = 1, \qquad Q_{M_2}([F_{-AB}], [F_{-AB}]) = -3.
\]
\end{lem}
\begin{proof}
Let $G = \langle f_g \rangle$. Then
\[
	H_2(M_2; \R)^G = \begin{cases}
	\R\{S_1 + S_2, \Sigma\} & \text{if } g = A, \\
	\R\{S_1-S_2, \Sigma\} & \text{if } g= -AB.
	\end{cases}
\]
The restriction of $Q_{M_2}$ to $H_2(M_2; \R)^G$ with respect to the $\R$-bases $(S_1+S_2, \Sigma)$ and $(S_1-S_2, \Sigma)$ respectively are
\[
	Q_{M_2}|_{H_2(M_2; \R)^G} = \begin{cases}
	\begin{pmatrix}
	2 & 0 \\
	0 & -1
	\end{pmatrix} & \text{if }g = A, \\
	\begin{pmatrix}
	-2 & 0 \\
	0 & -1
	\end{pmatrix} & \text{if }g = -AB.
	\end{cases}
\]
This shows that $\sigma(M_2/\langle f_A \rangle) = 0$ and $\sigma(M_2/\langle f_{-AB} \rangle) = -2$. Applying the $G$-signature theorem (Theorem \ref{thm:g-signature}) to both cases above shows
\[
	0 = -1 + Q_{M_2}([F_A],[F_A]) \qquad -4 = -1 + Q_{M_2}([F_{-AB}], [F_{-AB}]). \qedhere
\]
\end{proof}

Lemma \ref{lem:S-int-number} yields the appropriate homological obstructions to prove the following nonrealizability results.
\begin{prop}\label{prop:non-ex-1}
There is no lift of $\langle A, -B \rangle\cong(\Z/2\Z)^2 \leq \Mod(M_2)$ or of $\langle A, B \rangle\cong(\Z/2\Z)^2 \leq \Mod(M_2)$ to $\Diff^+(M_2)$.
\end{prop}
\begin{proof}
Suppose there is a lift $\langle f_A, f_{\pm B} \rangle \cong (\Z/2\Z)^2 \leq \Diff^+(M_2)$ with $[f_A] = A$ and $[f_{\pm B}] = \pm B$. Because $f_{\pm B}$ and $f_A$ commute, $f_{\pm B}$ restricts to a diffeomorphism on $F_A$, the unique $2$-dimensional connected component of $\Fix(f_A)$ by Lemma \ref{lem:fixed-A-AB}. Therefore $\pm B([F_A]) = [F_A]$ or $-[F_A]$. Observe that $H_2(M_2; \Z)$ has a direct sum decomposition 
\[
	H_2(M_2; \Z) = \Z\{\Sigma\} \oplus \Z\{ S_1, S_2\}
\]
into a sum of $(1)$- and $(-1)$-eigenspaces of of $\pm B$, meaning that either $[F_A] \in \Z\{\Sigma\}$ or $[F_A] \in \Z\{ S_1, S_2\}$.

By Lemma \ref{lem:S-int-number}, $Q_{M_2}([F_A], [F_A]) = 1$. If $[F_A] = a\Sigma$ for any $a \in \Z$ then $-a^2 = Q_{M_2}([F_A], [F_A]) = 1$ which is a contradiction. If $[F_A] = aS_1 + bS_2$ for any $a, b \in \Z$ then $2ab = Q_{M_2}([F_A], [F_A]) = 1$ which is also a contradiction.
\end{proof}

\begin{prop}\label{prop:non-ex-3}
There is no lift of $\langle -AB, -A \rangle \cong (\Z/2\Z)^2 \leq \Mod(M_2)$ to $\Diff^+(M_2)$.
\end{prop}
\begin{proof}
Suppose there is a lift $\langle f_{-AB}, f_{-A} \rangle \cong (\Z/2\Z)^2\leq \Diff^+(M_2)$ with $[f_{-AB}] = -AB$ and $[f_{-A}] = -A$. Because $f_{-AB}$ and $f_{-A}$ commute, $f_{-A}$ restricts to a diffeomorphism on $F_{-AB}$, which is the unique $2$-dimensional connected component of $\Fix(f_{-AB})$ by Lemma \ref{lem:fixed-A-AB}. There is a decomposition 
\[
	H_2(M_2; \Q) = \Q\{S_1-S_2\} \oplus \Q\{S_1 + S_2, \Sigma\}
\]
into $(1)$- and $(-1)$-eigenspaces of $-A$, which means that either $[F_{-AB}] \in \Q\{S_1-S_2\} \cap H_2(M_2; \Z)$ or $[F_{-AB}] \in \Q\{S_1 + S_2, \Sigma\} \cap H_2(M_2; \Z)$. 

If $[F_{-AB}] \in \Q\{S_1-S_2\} \cap H_2(M_2; \Z)$ then $[F_{-AB}] = a(S_1 - S_2)$ for some $a \in \Z$. Compute that 
\[
	-3 = Q_{M_2}([F_{-AB}], [F_{-AB}]) = -2a^2,
\]
which is a contradiction by Lemma \ref{lem:S-int-number}.

If $[F_{-AB}] \in \Q\{S_1 + S_2, \Sigma\} \cap H_2(M_2; \Z)$ then $[F_{-AB}] = a(S_1+S_2) + b\Sigma$ for some $a, b \in \Z$. Moreover, $H_2(M_2; \Q)$ has another direct sum decomposition 
\[
	H_2(M_2; \Q) \cong \Q\{S_1 - S_2, \Sigma\} \oplus \Q\{S_1 + S_2\}
\] 
into a sum of $(1)$- and $(-1)$-eigenspaces of $-AB$. Because $-AB([F_{-AB}]) = [F_{-AB}]$, 
\[
	[F_{-AB}] = a(S_1+S_2) + b\Sigma \in \Z\{S_1 - S_2, \Sigma\}
\]
which implies that $[F_{-AB}] = b\Sigma$. However, this is impossible since $-3 = Q_{M_2}([F_{-AB}], [F_{-AB}]) = -b^2$ by Lemma \ref{lem:S-int-number}.
\end{proof}

It turns out that with the exception of the subgroups of Propositions \ref{prop:non-ex-1} and \ref{prop:non-ex-3}, all other subgroups of $G_1$ are realizable. We explicitly construct the lifts of these subgroups of $G_1$ to $\Diff^+(M_2)$ in this next proposition.
\begin{prop}\label{prop:ex}
If $G \leq \Mod(M_2)$ is one of
\[
	\langle A, -I_3 \rangle, \, \langle B, -B \rangle, \, \langle AB, -I_3 \rangle, \, \text{or}\, \langle AB, -B\rangle
\]
then $G$ is realized by a complex equivariant connected sum.
\end{prop}
\begin{proof}
Define diffeomorphisms $h_A^1, h_B^1, h_{-I_3}^1, h_{AB}^1, h_{-B}^1 : \CP^1 \times \CP^1 \to \CP^1 \times \CP^1$ by:
\begin{align*}
h_A^1,\, h_{AB}^1: ([X_1:X_2],[Y_1:Y_2]) &\mapsto ([Y_1:Y_2], [X_1:X_2]), \\
h_B^1: ([X_1:X_2], [Y_1:Y_2]) &\mapsto ([-X_1:X_2], [Y_1:Y_2]), \\
h_{-I_3}^1,\, h_{-B}^1: ([X_1:X_2], [Y_1:Y_2]) &\mapsto ([\overline{X_1}:\overline{X_2}], [\overline{Y_1}:\overline{Y_2}]).
\end{align*}
The diffeomorphisms $h_A^1, h_B^1, h_{-I_3}^1, h_{AB}^1, h_{-B}^1$ have order $2$. On the other hand, define diffeomorphisms $h_A^2, h_B^2, h_{-I_3}^2, h_{AB}^2, h_{-B}^2: \CP^2 \to \CP^2$ by:
\begin{align*}
h_A^2, \, h_{-B}^2: [X:Y:Z] &\mapsto [-X:Y:Z], \\
h_B^2, \, h_{-I_3}^2: [X:Y:Z] &\mapsto [\overline X: \overline Y: \overline Z], \\
h_{AB}^2: [X:Y:Z] &\mapsto [-\overline X: \overline Y: \overline Z].
\end{align*}
The fixed sets of the diffeomorphisms defined above are: 
\begin{align*}
\Fix(h_A^1) = \Fix(h_{AB}^1) &= \{([X_1:X_2], [X_1:X_2])\} \cong \CP^1, \\
\Fix(h_B^1) &= \{([0:1],[Y_1:Y_2]) \} \sqcup \{([1:0],[Y_1:Y_2])\} \cong \CP^1 \sqcup \CP^1, \\
\Fix(h_{-I_3}^1) = \Fix(h_{-B}^1) &=  \{([a:b], [c:d]) \in \RP^1 \times \RP^1 \} \cong \mathbb T^2, \\
\Fix(h_A^2) = \Fix(h_{-B}^2) &= \{[0:Y:Z] \in \CP^2\} \sqcup \{[1:0:0]\} \cong \CP^1 \sqcup \{p\}, \\
\Fix(h_B^2) = \Fix(h_{-I_3}^2) &= \{[a:b:c]: a, b, c\in \R\} \cong \RP^2,\\
\Fix(h_{AB}^2)&= \{[ai : b: c] : a, b, c\in \R\} \cong \RP^2.
\end{align*}

\begin{enumerate}
	\item Compute that 
	\[
	\langle h_A^1, h_{-I_3}^1 \rangle \cong (\Z/2\Z)^2 \leq \Diff^+(\CP^1 \times \CP^1)\quad\text{and}\quad \langle h_A^2, h_{-I_3}^2 \rangle \cong (\Z/2\Z)^2 \leq \Diff^+(\overline{\CP^2}).
	\]
	Let $q_1 = ([0:1], [0:1]) \in \Fix(h_A^1) \cap \Fix(h_{-I_3}^1)$ and $q_2 = [0:0:1] \in \Fix(h_A^2) \cap \Fix(h_{-I_3}^2)$. The connected component $F_{q_1}$ of $q_1$ in $\Fix(h_A^1) \cap \Fix(h_{-I_3}^1)$ is
	\[
		F_{q_1} = \{([a:b],[a:b]) : a, b \in \R\} \cong \mathbb S^1.
	\]
	The connected component $F_{q_2}$ of $q_2$ in $\Fix(h_A^2) \cap \Fix(h_{-I_3}^2)$ is
	\[
		F_{q_2} = \{[0:a:b] : a, b\in \R\} \cong \mathbb S^1.
	\]
	Therefore, apply Lemma \ref{lem:glue} to see that $\langle h_A^1 \# h_A^2, h_{-I_3}^1 \# h_{-I_3}^2 \rangle \leq \Diff^+(M_2)$ is a realization of $\langle A, -I_3 \rangle$ by a complex equivariant connected sum.

	\item Compute that 
	\[
	\langle h_B^1, h_{-B}^1 \rangle \cong (\Z/2\Z)^2 \leq \Diff^+(\CP^1 \times \CP^1)\quad\text{and}\quad \langle h_B^2, h_{-B}^2 \rangle \cong (\Z/2\Z)^2 \leq \Diff^+(\overline{\CP^2}).
	\]
	Let $q_1 = ([0:1], [0:1]) \in \Fix(h_B^1) \cap \Fix(h_{-B}^1)$ and $q_2 = [0:0:1] \in \Fix(h_B^2) \cap \Fix(h_{-B}^2)$. The connected component $F_{q_1}$ of $q_1$ in $\Fix(h_B^1) \cap \Fix(h_{-B}^1)$ is
	\[
		F_{q_1} = \{([0:1], [a:b]) : a, b\in \R\} \cong \mathbb S^1.
	\]
	The connected component $F_{q_2}$ of $q_2$ in $\Fix(h_B^2) \cap \Fix(h_{-B}^2)$ is
	\[
		F_{q_2} =  \{[0:a:b] : a, b\in\R\} \cong \mathbb S^1.
	\]
	Therefore, apply Lemma \ref{lem:glue} to see that $\langle h_B^1 \# h_B^2, h_{-B}^1 \# h_{-B}^2 \rangle \leq \Diff^+(M_2)$ is a realization of $\langle B, -B \rangle$ by a complex equivariant connected sum.

	\item Compute that 
	\[
	\langle h_{-I_3}^1, h_{AB}^1 \rangle \cong (\Z/2\Z)^2 \leq \Diff^+(\CP^1 \times \CP^1)\quad\text{and}\quad\langle h_{-I_3}^2, h_{AB}^2 \rangle \cong (\Z/2\Z)^2 \leq \Diff^+(\overline{\CP^2}).
	\]
	Let $q_1 = ([0:1],[0:1]) \in \Fix(h_{-I_3}^1) \cap \Fix(h_{AB}^1)$ and $q_2 = [0:0:1] \in \Fix(h_{-I_3}^2) \cap \Fix(h_{AB}^2)$. The connected component $F_{q_1}$ of $q_1$ in $\Fix(h_{-I_3}^1) \cap \Fix(h_{AB}^1)$ is
	\[
		F_{q_1} = \{([a:b],[a:b]) : a, b\in \R\} \cong \mathbb S^1.
	\]
	The connected component $F_{q_2}$ of $q_2$ in $\Fix(h_{-I_3}^2) \cap \Fix(h_{AB}^2)$ is
	\[
		F_{q_2} = \{[0:a:b] : a, b\in \R\} \cong \mathbb S^1.
	\]
	Therefore, apply Lemma \ref{lem:glue} to see that $\langle h_{-I_3}^1 \# h_{-I_3}^2, h_{AB}^1 \# h_{AB}^2\rangle \leq \Diff^+(M_2)$ is a realization of $\langle -I_3, AB \rangle$ by a complex equivariant connected sum.

	\item Compute that
	\[
	\langle h_{AB}^1, h_{-B}^1 \rangle \cong (\Z/2\Z)^2 \leq \Diff^+(\CP^1 \times \CP^1)\quad\text{and}\quad\langle h_{AB}^2, h_{-B}^2 \rangle \cong (\Z/2\Z)^2 \leq \Diff^+(\overline{\CP^2}).
	\]
	Let $q_1 = ([0:1],[0:1]) \in \Fix(h_{AB}^1) \cap \Fix(h_{-B}^1)$ and $q_2 = [0: 0: 1] \in \Fix(h_{AB}^2) \cap \Fix(h_{-B}^2)$. The connected component $F_{q_1}$ of $q_1$ in $\Fix(h_{AB}^1) \cap \Fix(h_{-B}^1)$ is
	\[
		F_{q_1} = \{([a:b],[a:b]) : a, b\in \R\} \cong \mathbb S^1.
	\]
	The connected component $F_{q_2}$ of $q_2$ in $\Fix(h_{AB}^2) \cap \Fix(h_{-B}^2)$ is
	\[
		F_{q_2} = \{[0:a:b] : a, b\in \R\} \cong \mathbb S^1.
	\]
	Therefore, apply Lemma \ref{lem:glue} to see that $\langle h_{AB}^1 \# h_{AB}^2, h_{-B}^1\# h_{-B}^2 \rangle \leq \Diff^+(M_2)$ is a realization of $\langle AB, -B \rangle$ by a complex equivariant connected sum. \qedhere
\end{enumerate}
\end{proof}

\subsubsection{$G_2 = \langle \Phi, \Psi, -I_3 \rangle \cong D_4 \times \Z/2\Z$}\label{sec:g2}

In this section we show that the subgroup $G_2$ is realizable by diffeomorphisms of $M_2$ given by a complex equivariant connected sum. 

\begin{prop}\label{prop:Z2-D4}
The group $G_2 \cong D_4 \times (\Z/2\Z) \leq \Mod(M_2)$ is realized by a complex equivariant connected sum. 
\end{prop}
\begin{proof}
Define diffeomorphisms $h_{\Phi}, h_{\Psi}, h_{-I_3}: \CP^2 \to \CP^2$ by
\begin{align*}
h_\Phi([X:Y:Z]) &= [-Y:X:Z],\\
h_\Psi([X:Y:Z]) &= [X:-Y:Z],\\
h_{-I_3}([X:Y:Z]) &= [\overline X: \overline Y : \overline Z].
\end{align*}
Let $p_1 = [1:0:0] \in \CP^2$ and $p_2 = [0:1:0]\in \CP^2$. The subgroup $\langle h_\Phi, h_\Psi, h_{-I_3} \rangle \leq \Diff^+(\CP^2)$ preserves the set $\{p_1, p_2\} \subseteq \CP^2$. It is straightforward to check that $\langle h_\Phi, h_\Psi, h_{-I_3} \rangle \cong D_4 \times (\Z/2\Z)$. 

On the other hand, let $N_1, N_2\cong \overline{\CP^2}$. Define $g_\Phi: N_1 \sqcup N_2 \to N_1 \sqcup N_2$ by
\begin{align*}
	g_\Phi|_{N_1}([X:Y:Z]) = [\overline X: -\overline Y: \overline Z]\in N_2, \\
	g_\Phi|_{N_2}([X:Y:Z]) = [X:Y:-Z]\in N_1.
\end{align*}
Define $g_\Psi, g_{-I_3}: N_1 \sqcup N_2 \to N_1\sqcup N_2$ 
\begin{align*}
g_\Psi([X:Y:Z]) &= \begin{cases}
[\overline X : -\overline Y: \overline Z] \in N_1 & \text{if }[X:Y:Z] \in N_1, \\
[X:Y:-Z] \in N_2 & \text{if }[X:Y:Z] \in N_2,
\end{cases} \\
g_{-I_3}([X:Y:Z]) &= \begin{cases}
[\overline X : \overline Y : \overline Z]\in N_1 &\text{if }[X:Y:Z] \in N_1, \\
[\overline X : \overline Y : \overline Z]\in N_2 & \text{if }[X:Y:Z] \in N_2.
\end{cases}
\end{align*}
The maps $g_\Phi, g_\Psi, g_{-I_3}$ preserve the set $\{q_1, q_2\} \subseteq N_1 \sqcup N_2$ with $q_1 = [0:0:1] \in N_1$ and $q_2 = [0:0:1] \in N_2$. It is straightforward to check that $\langle g_{\Phi}, g_\Psi, g_{-I_3}\rangle \cong D_4 \times (\Z/2\Z)$.

Let $B_k = \{(a,b,c,d) \in \R^4 : a^2 + b^2 + c^2 + d^2 < 1\}$ for $k = 1, 2$. Define smooth embeddings $i_k: B_k \to \CP^2$, $j_k: B_k \to N_k$ for $k = 1, 2$ by
\begin{align*}
i_1(a, b, c, d) &= [1 : a+bi : c+di] \in \CP^2, & j_1(a, b, c, d) &= [c+bi : a-di : 1] \in N_1, \\
i_2(a, b, c, d) &= [a+bi : 1 : c+di] \in \CP^2, & j_2(a, b, c, d) &= [c+bi : a-di : 1] \in N_2.
\end{align*}
The embeddings $i_1, i_2$ are orientation-preserving while $j_1, j_2$ are orientation-reversing. Then we explicitly identify $M_2$ with 
\[
	(\CP^2 - \{p_1, p_2\}) \sqcup (N_1 \sqcup N_2 -\{q_1, q_2\}) / \sim
\]
with $i_k(tu_k) \sim j_k((1-t)u_k)$ for all $t \in (0, 1)$ and $u_k \in \partial B_k$ for $k = 1, 2$. For $c = \Phi, \Psi,$ and $-I_3$, let $f_c \in \Diff^+(M_2)$ be defined
\[
	f_c(x) = \begin{cases}
	h_c(x) & x \in \CP^2 - \{p_1, p_2\} \\
	g_c(x) & x \in (N_1 \sqcup N_2) - \{q_2, q_2\}. 
	\end{cases}
\]
\begin{figure}
\centering
\includegraphics[width=0.9\textwidth]{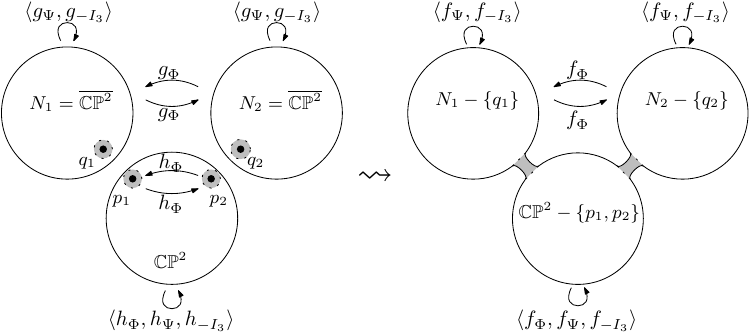}
\caption{The complex equivariant connected sum $(\CP^2 \#2\overline{\CP^2}, G_2)$ constructed in the proof of Proposition \ref{prop:Z2-D4}. Left: The group $G_2$ acts on $\CP^2$ and $2 \overline{\CP^2}$. The subgroup $\langle \Phi^2, \Psi, -I_3\rangle$ fixes the points $p_1$, $p_2$, $q_1$, and $q_2$ and preserves some neighborhoods (in grey) of each point. The group $G_2$ acts on the union of these neighborhoods. Right: A connected sum formed by gluing the neighborhood of $q_i$ to the neighborhood of $p_i$ (in grey) for $i = 1, 2$ in a $G_2$-equivariant way.} \label{fig:cecs-d4-z2}
\end{figure}
See Figure \ref{fig:cecs-d4-z2} for an illustration of the complex equivariant connected sum $(M_2, \langle f_\Phi, f_\Psi, f_{-I_3} \rangle)$. We now show that each $f_c$ is a well-defined diffeomorphism of $M_2$. Let $x = (a,b,c,d) \in \partial B_4$ and $t \in (0, 1)$. Compute for $i_1(tx) = [1 : t(a+bi): t(c+di)] \in \image(i_1) \subseteq \CP^2 - \{p_1, p_2\}$,
\begin{align*}
h_\Phi(i_1(tx))&= i_2(t(-a, -b, c, d)) = j_2((1-t)(-a,-b,c,d)) = g_\Phi(j_1((1-t)x)), \\
h_\Psi(i_1(tx)) &= i_1(t(-a, -b, c, d)) = j_1((1-t)(-a, -b, c, d)) = g_\Psi(j_1((1-t)x)), \\
h_{-I_3}(i_1(tx)) &= i_1(t(a, -b, c, -d)) = j_1((1-t)(a, -b, c, -d)) = g_{-I_3}(j_1((1-t)x)).
\end{align*}
For $i_2(tx) = [t(a+bi) : 1 : t(c+di)] \in \image(i_2) \subseteq \CP^2 - \{p_1, p_2\}$,
\begin{align*}
h_\Phi(i_2(tx)) &= i_1(t(-a, -b, -c, -d)) = j_1((1-t)(-a,-b,-c,-d)) = g_\Phi(j_2((1-t)x)), \\
h_\Psi(i_2(tx)) &= i_2(t(-a, -b, -c, -d) = j_2((1-t)(-a,-b,-c,-d)) = g_\Psi(j_2((1-t)x)), \\
h_{-I_3}(i_2(tx)) &= i_2(t(a, -b, c, -d)) = j_2((1-t)(a,-b,c,-d)) = g_{-I_3}(j_2((1-t)x)).
\end{align*}
We have computed explicitly that $h_c(tx) \sim g_c((1-t)x)$ for all $c = \Phi, \Psi, -I_3$, $x \in \partial B_k$ for $k = 1, 2$ and $t \in (0, 1)$. Therefore, $\langle f_\Phi, f_\Psi, f_{-I_3} \rangle \cong D_4 \times (\Z/2\Z) \leq \Diff^+(M_2)$ is a desired lift of $\langle \Phi, \Psi, -I_3 \rangle \cong D_4 \times (\Z/2\Z) \leq \Mod(M_2)$ with $[f_\Phi] = \Phi$, $[f_\Psi] = \Psi$, and $[f_{-I_3}] = -I_3$ by construction. 

Finally, let $H = \langle \Phi^2, \Psi, -I_3 \rangle \trianglelefteq G_2$. As constructed above, $H$ acts on both $\CP^2$ and $\overline{\CP^2}$ smoothly, and $H$ is the stabilizer of $p_1 \in \CP^2$ and $q_1 \in N_1 \cong \overline{\CP^2}$. Hence $(M_2, \langle f_\Phi, f_\Psi, f_{-I_3} \rangle)$ is an explicit construction of the complex equivariant connected sum $(\CP^2 \# (G_2 \times_H \overline{\CP^2}), G_2)$. 
\end{proof}

\subsubsection{Proof of Theorem \ref{thm:m2} and its corollaries}\label{sec:proof-m2}
The following proposition ties together the lemmas of the previous sections.
\begin{prop}\label{prop:m2}
Let $G \leq \Mod(M_2)$ be a finite subgroup. There exists a lift $\tilde G$ of $G$ to $\Diff^+(M_2)$ under $\pi: \Homeo^+(M_2) \to \Mod(M_2)$ if and only if $G$ is conjugate in $\Mod(M_2)$ to a subgroup of:
\begin{enumerate}
	\item $\left\langle \Phi, \Psi, -I_3  \right\rangle \cong D_4 \times (\Z/2\Z)$, or 
	\item one of the four groups $\langle A, -I_3 \rangle$, $\langle B, -I_3  \rangle$, $\langle AB, -I_3 \rangle$, or $\langle -A, -B \rangle$, each isomorphic to $(\Z/2\Z)^2$.
\end{enumerate}
\end{prop}
\begin{proof} 
By Lemma \ref{lem:max-finite-subgroups}, any finite subgroup $G \leq \Mod(M_2)$ is conjugate in $\Mod(M_2)$ to a subgroup of 
\begin{enumerate}
\item $\langle G_{H-E_1-E_2}, -I_3 \rangle = \langle \RRef_{E_1-E_2}, \RRef_{E_2}, -I_3 \rangle = \langle \Phi, \Psi, -I_3 \rangle$ or
\item $\langle G_{E_2}, -I_3 \rangle = \langle \RRef_{H-E_1-E_2}, \RRef_{E_1-E_2}, -I_3 \rangle = \langle A, B, -I_3 \rangle$.
\end{enumerate}
Proposition \ref{prop:Z2-D4} constructs a lift of $\langle \Phi, \Psi, -I_3 \rangle$ to $\Diff^+(M_2)$ explicitly. Suppose $G$ is a subgroup of $\langle A, B, -I_3\rangle \cong (\Z/2\Z)^3$ and $G$ has a lift $\tilde G \leq \Diff^+(M_2)$. Because there are subgroups of $\langle A, B, -I_3 \rangle$ which do not have lifts to $\Diff^+(M_2)$ (Propositions \ref{prop:non-ex-1} and \ref{prop:non-ex-3}), $G$ must be a proper subgroup of $\langle A, B, -I_3 \rangle$. Observe that any two out of the seven nontrivial elements of $\langle A, B, -I_3 \rangle$ determine a subgroup of order $4$ in $\langle A, B, -I_3 \rangle$, and there are three possible sets of generators for each order $4$ subgroup. This gives seven distinct subgroups of order $4$.

Propositions \ref{prop:non-ex-1} and \ref{prop:non-ex-3} account for three subgroups isomorphic to $(\Z/2\Z)^2$ that cannot be realized by diffeomorphisms. Proposition \ref{prop:ex} gives lifts for the four remaining subgroups isomorphic to $(\Z/2\Z)^2$; all of these lifts are realized as complex equivariant connected sums. This accounts for all maximal proper subgroups of $\langle A, B, -I_3\rangle$. On the other hand, all nontrivial subgroups of maximal proper subgroups of $\langle A, B, -I_3 \rangle$ are cyclic of order $2$. Note that all order $2$ elements of $\langle A, B, -I_3 \rangle$ are contained in some subgroup of order $4$, which have lifts to $\Diff^+(M_2)$, listed in Proposition \ref{prop:ex}.

Finally, the map $\pi|_{\Diff^+(M_2)}: \Diff^+(M_2) \to \Mod(M_2)$ is surjective by Theorem \ref{thm:diffeo-realizable}. Therefore, for any finite subgroup $H = gGg^{-1} \subseteq \Mod(M_2)$ with $G$ a finite subgroup of $\langle A, B, -I_3 \rangle$ or $\langle \Phi, \Psi, -I_3 \rangle$, let $\alpha \in \Diff^+(M_2)$ be a representative of $g \in \Mod(M_2)$. Then if a lift $\tilde G$ of $G$ to $\Diff^+(M_2)$ exists then $\tilde H = \alpha \tilde G \alpha^{-1}$ is a lift of $H$ and vice versa. Therefore, a lift $\tilde G$ of $G$ to $\Diff^+(M_2)$ exists if and only if $G$ is conjugate to a subgroup given in the statement of the proposition.
\end{proof}

We can now prove Theorem \ref{thm:m2}.
\begin{proof}[Proof of Theorem \ref{thm:m2}]
All constructions of Proposition \ref{prop:ex} and Proposition \ref{prop:Z2-D4}, which account for all realizable cases of Proposition \ref{prop:m2}, are given as complex equivariant connected sums. These account for all finite subgroups of $\Mod(M_2)$ which lift to $\Diff^+(M_2)$.
\end{proof}

Finally, we deduce the corollaries of Theorem \ref{thm:m2}.
\begin{proof}[Proof of Corollary \ref{cor:finite-order-m2}]
By Lemma \ref{lem:max-finite-subgroups}, any finite order element $c$ of $\Mod(M_2)$ is  conjugate to some element contained in $G_1$ or $G_2$. Compute that all individual finite order elements of $G_1$ and $G_2$ are contained in some subgroup given in Proposition \ref{prop:m2} that lifts to $\Diff^+(M_2)$. Again since $\pi|_{\Diff^+(M_2)}: \Diff^+(M_2) \to \Mod(M_2)$ is surjective, this is enough to conclude that $c$ is represented by an order $n$ diffeomorphism in $\Diff^+(M_2)$.
\end{proof}

\begin{proof}[Proof of Corollary \ref{cor:dehn-twist-m2}]
The mapping class of a Dehn twist about a $(-2)$-sphere in $M_2$ has order $2$ in $\Mod(M_2)$. The corollary is a special case of Corollary \ref{cor:finite-order-m2}.
\end{proof}

\begin{proof}[Proof of Corollary \ref{cor:non-realizable-subgroup}]
By Propositions \ref{prop:non-ex-1} and \ref{prop:non-ex-3}, there exist subgroups $G \leq \Mod(M_2)$ isomorphic to $(\Z/2\Z)^2$ that cannot be realized by diffeomorphisms. By Corollary \ref{cor:finite-order-m2}, all elements $g \in G$ are realized by a complex equivariant connected sum.
\end{proof}

\subsection{Smooth Nielsen realization problem for $M_3$}\label{sec:m3}
In this section we use similar techniques as used in the proof of Theorem \ref{thm:m2} to show that only one of the three maximal finite subgroups of $\Mod(M_3)$ lifts to $\Diff^+(M_3)$.

Throughout this section let
\[
	\sigma(ij) = \RRef_{E_i-E_j} \quad\quad R(k) = \RRef_{E_k}, \quad\quad \psi = \RRef_{H-E_1-E_2-E_3}, \quad\quad c = R(1)R(2)R(3).
\]
For any subgroup $G \leq \Mod(M_3)$ with a lift to $\Diff^+(M_3)$, we denote this lift by $\langle f_g : g \in G \rangle\leq \Diff^+(M_3)$ where $f_g \in \Diff^+(M_3)$ and $[f_g] = g \in \Mod(M_3)$. The following is a reformulation of \cite[Corollary 2.6]{edmonds}.
\begin{lem}[{Edmonds, \cite[Corollary 2.6]{edmonds}}]\label{lem:nonzero-homology-2-dim}
Let $M$ be a simply-connected, closed, oriented $4$-manifold. Let $f: M \to M$ be an orientation-preserving diffeomorphism of prime order $p$. If $\Fix(f)$ has two or more connected components and $S \subseteq \Fix(f)$ is an orientable $2$-dimensional component then $[S] \neq 0 \in H_2(M; \Z)$.
\end{lem}
\begin{proof}
If $\Fix(f)$ is not purely $2$-dimensional then \cite[Corollary 2.6]{edmonds} says that the $2$-dimensional components represent linearly independent elements of $H_2(M; \Z/p\Z)$. Then because $[S] \neq 0 \in H_2(M;\Z/p\Z)$ and $S$ is an orientable surface, $[S] \neq 0\in H_2(M; \Z)$.

If $\Fix(f)$ is purely $2$-dimensional with $k$-many connected components then \cite[Corollary 2.6]{edmonds} says that any $(k-1)$-many connected components represent linearly independent elements of $H_2(M; \Z/p\Z)$. By the assumption that $k > 1$, this implies that $[S] \neq 0 \in H_2(M; \Z/p\Z)$. Because $S$ is an orientable surface, $[S] \neq 0 \in H_2(M; \Z)$.
\end{proof}

In addition to the homological results used in Section \ref{sec:m2}, a main idea of many arguments in this section is the following: suppose a finite abelian group $G$ acts smoothly on some $4$-manifold $M$ and fix a $G$-invariant metric on $M$. For all $g, h \in G$, the fact that $g, h$ commute implies that $h$ restricts to a diffeomorphism on $\Fix(g)$ and vice versa. If there is a unique isolated fixed point $p \in \Fix(g)$ then $p$ must be fixed by all elements of $G$ since diffeomorphisms of $\Fix(g)$ must send $0$-dimensional components to $0$-dimensional components and so the tangential representation $G \to \SO(T_pM)$ is faithful.
\begin{lem}\label{lem:no-isolated-points-12-23}
If $G = \langle \sigma(12), R(3) \rangle \cong (\Z/2\Z)^2 \leq \Mod(M_3)$ lifts to $\Diff^+(M_3)$ then $\Fix(f_{\sigma(12)}) \cong F_1 \sqcup F_2$ with $F_i \cong \mathbb S^2$ for $i = 1, 2$ and $\Fix(f_{\sigma(12)R(3)}) \cong \RP^2 \sqcup \{p\}$. Moreover, $Q_{M_3}([F_1], [F_1]) + Q_{M_3}([F_2], [F_2]) = 0$. 
\end{lem}
\begin{proof}
Suppose a lift $\langle f_g : g \in G \rangle \leq \Diff^+(M_3)$ exists. By Remark \ref{rmk:euler-char}, $\chi(\Fix(f_{\sigma(12)})) = 4$ and $\chi(\Fix(f_{\sigma(12)R(3)})) = 2$; therefore, $\Fix(f_{\sigma(12)}) \neq \emptyset$ and $\Fix(f_{\sigma(12)R(3)}) \neq \emptyset$. By Proposition \ref{prop:edmonds}, $\beta_1(\Fix(f_{\sigma(12)})) = 0$ and $\beta_0(\Fix(f_{\sigma(12)})) + \beta_2(\Fix(f_{\sigma(12)})) = 4$. Again by Proposition \ref{prop:edmonds}, $\beta_1(\Fix(f_{\sigma(12)R(3)})) = 1$ and $\beta_0(\Fix(f_{\sigma(12)R(3)})) + \beta_2(\Fix(f_{\sigma(12)R(3)})) = 3$. Combining with the fact that each $\Fix(\varphi)$ is a disjoint union of surfaces and isolated points shows that 
\[
	\Fix(f_{\sigma(12)}) \cong [4], \, \mathbb S^2 \sqcup [2],\, \text{or }\mathbb S^2 \sqcup \mathbb S^2, \quad\quad \Fix(f_{\sigma(12)R(3)}) \cong \RP^2 \sqcup \{p\}.
\]
Because $p \in \Fix(f_{\sigma(12)R(3)})$ is a unique isolated fixed point, $p$ is fixed by $f_g$ for all $g \in G$ and the tangential representation $G \to \SO(T_pM)$ is faithful. If $p$ is an isolated fixed point of $f_h$ for any $h \in G$, then $d(f_h)_p = -I_4$ which is the only element of $\SO(4)$ of order $2$ that does not fix any nonzero vector. Therefore, $p$ is not an isolated fixed point of $f_h$ for any $h \neq \sigma(12) R(3)\in G$.

Because $f_{\sigma(12)}$ acts on $\Fix(f_{\sigma(12)R(3)})$, it must fix the isolated point $p \in \Fix(f_{\sigma(12)R(3)})$. Then $p$ cannot be an isolated fixed point of $f_{\sigma(12)}$, and so $\Fix(f_{\sigma(12)}) \not\cong [4]$.

Suppose $\Fix(f_{\sigma_{(12)}}) \cong \mathbb S^2 \sqcup [2]$. By the $G$-signature theorem (Theorem \ref{thm:g-signature}) for $G = \langle f_{\sigma(12)} \rangle$ with $S \cong \mathbb S^2 \subseteq \Fix(f_{\sigma(12)})$,
\[
	-2 = -2 + Q_{M_3}([S],[S]).
\]
The diffeomorphism $f_{R(3)}$ commutes with $f_{\sigma(12)}$ and so $f_{R(3)}$ restricts to a diffeomorphism of $S$ and $R(3)([S]) = \pm [S]$.

If $R(3)([S]) = -[S]$ then $[S] = aE_3 \in \Z\{E_3\}$ for some $a \in \Z$. Then $Q_{M_3}([S], [S]) = -a^2 = 0$ implies that $a = 0$. By Lemma \ref{lem:nonzero-homology-2-dim}, this is a contradiction.

If $R(3)([S]) = [S]$ then $[S] \in \Z\{H, E_1 + E_2\}$ since $\sigma_{(12)}([S]) = [S]$. Hence $[S] = a H + b(E_1 + E_2)$ for some $a, b\in \Z$, and so $Q_{M_3}([S], [S]) = a^2 - 2b^2 = 0$. The only integral solution to this equation is $(a,b) = (0,0)$, meaning $[S] = 0$. This is again a contradiction by Lemma \ref{lem:nonzero-homology-2-dim}, and so $\Fix(f_{\sigma(12)}) \not\cong \mathbb S^2 \sqcup [2]$.

Finally, the only remaining choice is $\Fix(f_{\sigma(12)}) \cong \mathbb S^2 \sqcup \mathbb S^2$. Denote these $2$-spheres by $F_1$ and $F_2$. The last claim follows because $-2 = -2 + Q_{M_3}([F_1], [F_1]) + Q_{M_3}([F_2], [F_2])$ by the $G$-signature theorem (Theorem \ref{thm:g-signature}) for $G = \langle f_{\sigma(12)} \rangle$.
\end{proof}

\begin{prop}\label{prop:non-realizable-1}
There is no lift of $G = \langle \sigma(12), \sigma(23), R(3) \rangle \leq \Mod(M_3)$ to $\Diff^+(M_3)$.
\end{prop}
\begin{proof}
Suppose a lift $\langle f_g : g \in G \rangle \leq \Diff^+(M_3)$ exists. By Lemma \ref{lem:no-isolated-points-12-23}, $\Fix(f_{\sigma(12)R(3)}) \cong \RP^2 \sqcup \{q\}$ for some isolated fixed point $q \in M_3$ of $f_{\sigma(12)R(3)}$. Note that $c = R(1) R(2) R(3)$ is in the center $Z(G)$ of $G$, so $f_c$ must fix the point $q$. Because $\Fix(f_c) \neq \emptyset$, Proposition \ref{prop:edmonds} shows that $\beta_1(\Fix(f_c)) = 3$ and $\beta_0(\Fix(f_c)) + \beta_2(\Fix(f_c)) = 3$. Combining with the fact that each $\Fix(f_c)$ is a disjoint union of surfaces and isolated points shows that $\Fix(f_c) \cong \#3 \RP^2 \sqcup \{p\}$ for some isolated fixed point $p \in M_3$ of $f_c$. All elements $g \in G$ commute with $c$, and so $f_g$ acts on $\Fix(f_c)$ by diffeomorphism. This shows that $p$ is a fixed point of $f_g$ for all $g \in G$ and the tangential representation $G \to \SO(T_pM_3)$ is faithful. Because $f_c$ has order $2$, its derivative at the isolated fixed point $p$ must be $d(f_c)_p =  -I_4$. Therefore, $p$ cannot be an isolated fixed point of the lift of $f_h$ for any $h \in G$ of order $2$.

If $p$ is an isolated point in $\Fix(f_{\sigma(12)}) \cap \Fix(f_{R(3)})$ then $p$ is an isolated fixed point of $f_{\sigma(12) R(3)}$, which is a contradiction. So $p$ is in a $1$-dimensional component of $\Fix(f_{\sigma(12)}) \cap \Fix(f_{R(3)})$. Let $\Fix(f_{\sigma(12)}) = F_1 \sqcup F_2$ with $F_i \cong \mathbb S^2$ by Lemma \ref{lem:no-isolated-points-12-23}. Because $R(3)$ and $\sigma(12)$ commute in $G$, the diffeomorphism $f_{R(3)}$ acts on $\Fix(f_{\sigma(12)})$. Assuming $p \in F_1$, we must have $R_3([F_1]) = -[F_1]$, i.e. $[F_1] = aE_3$ for some $a \in \Z$, because $f_{R(3)}$ only fixes a $1$-dimensional submanifold of $F_1$. On the other hand, $f_c$ acts by a diffeomorphism on $F_1$ since $f_c(p) = p$ and $c$ commutes with $\sigma(12)$. Therefore $c([F_1]) = c(aE_3) = -[F_1]$. However, $d(f_c)_p$ must act by $-I_2$ on $T_{p} S_1$ because $d(f_c)_p = -I_4$ on $T_{p} M_3$. This means that $f_c$ acts in an orientation-preserving way on $F_1$, i.e. $c([F_1]) = [F_1]$. Therefore, $[F_1] = 0$. This is a contradiction by Lemma \ref{lem:nonzero-homology-2-dim}.
\end{proof}

\begin{prop}\label{prop:non-realizable-2}
There is no lift of $G = \langle \psi, \sigma(12), R(3) \rangle\leq \Mod(M_3)$ to $\Diff^+(M_3)$.
\end{prop}
\begin{proof}
Suppose there is such a lift $\langle f_g: g \in G \rangle$ to $\Diff^+(M_3)$. Observe that $\sigma(12) \in Z(G)$ and $\Fix(f_{\sigma(12)}) = F_1 \sqcup F_2$ with $F_i \cong \mathbb S^2$ and $Q_{M_3}([F_1], [F_1]) + Q_{M_3}([F_2], [F_2]) = 0$ by Lemma \ref{lem:no-isolated-points-12-23}. Lemma \ref{lem:no-isolated-points-12-23} shows that $\Fix(f_{\sigma(12)R(3)}) \cong \RP^2 \sqcup \{p\}$. Since $f_{\sigma(12)}$ and $f_{R(3)}$ act on $\Fix(f_{\sigma(12)R(3)})$, they must fix the unique isolated point $p$. Because $f_{\sigma(12)}$ and $f_{R(3)}$ have a common fixed point, $f_{R(3)}$ preserves each $F_i$. Moreover, $p$ must be an isolated point in $\Fix(f_{\sigma(12)}) \cap \Fix(f_{R(3)})$, since otherwise $p$ would not be an isolated fixed point of $f_{\sigma(12)R(3)}$. If $F_1$ is the component of $\Fix(f_{\sigma(12)})$ containing $p$ then $d(f_{R(3)})_p|_{T_pS_1} = -I_2$, i.e. $f_{R(3)}$ acts by an orientation-preserving diffeomorphism on $F_1$ and $R(3)([F_1]) = [F_1]$. Because $\sigma(12)([F_1]) = [F_1]$, we know that $[F_1] \in \Z\{H,E_1 + E_2\}$.

Next, note that $\psi$ and $\sigma(12)$ commute and consider the action of $f_{\psi}$ on $\Fix(\sigma(12)) = F_1 \sqcup F_2$. With respect to the $\Z$-basis $(H-E_1, H-E_2, H-E_1-E_2, E_3)$ of $H_2(M_3; \Z)$,
\[
	\psi = \begin{pmatrix}
	1 & 0 & 0 & 0 \\
	0 & 1 & 0 & 0 \\
	0 & 0 & 0 & 1 \\
	0 & 0 & 1 & 0
	\end{pmatrix}.
\]
\begin{enumerate}
	\item Suppose $\psi([F_1]) = \pm [F_2]$. Write $[F_1] = aH + b(E_1 + E_2)$ for some $a, b\in \Z$. Then 
	\[
		Q_{M_3}([F_1], [F_1]) + Q_{M_3}([F_2],[F_2]) = 2(a^2 - 2b^2) = 0
	\]
	since 
	\[
		Q_{M_3}([F_2],[F_2]) = Q_{M_3}(\psi([F_1]), \psi([F_1]))= Q_{M_3}([F_1], [F_1]) = a^2 - 2b^2.
	\]
	The only integral solution $(a,b) = (0,0)$, a contradiction by Lemma \ref{lem:nonzero-homology-2-dim}.

	\item If $\psi([F_1]) = -[F_1]$ then $[F_1] \in \Z\{H - E_1 - E_2 - E_3\}$. Compute that 
	\[
		\Z\{H, E_1 + E_2\} \cap \Z\{H - E_1 - E_2 - E_3\} = 0,
	\] 
	which implies that $[F_1] = 0$. This is a contradiction again by Lemma \ref{lem:nonzero-homology-2-dim}.

	\item If $\psi([F_1]) = [F_1]$ then $[F_1] \in \Z\{H-E_1-E_2+E_3, H-E_1, H-E_2\}$. Compute that 
	\[
		\Z\{H-E_1-E_2+E_3, H-E_1, H-E_2\} \cap \Z\{H, E_1 + E_2\} = \Z\{2H - E_1 - E_2\}.
	\]
	So $[F_1] = c(2H - E_1 - E_2)$ for some $c \in \Z$ and $Q_{M_3}([F_1], [F_1]) = 2c^2$. 

	Next, note that $R(3)([F_2]) = \pm [F_2]$ because $R(3)([F_1]) = [F_1]$ and $f_{R(3)}$ acts on $\Fix(f_{\sigma(12)})$. Because $\sigma(12)([F_2]) = [F_2]$, we know that $[F_2] \in \Z\{H, E_1+E_2, E_3\}$.
	\begin{enumerate}
		\item If $R(3)([F_2]) = -[F_2]$ then $[F_2] \in \Z\{E_3\}$. So $Q_{M_3}([F_2], [F_2]) = -d^2$ for some $d \in \Z$. Then
		\[
			Q_{M_3}([F_1], [F_1]) + Q_{M_3}([F_2],[F_2]) = 2c^2 - d^2 = 0
		\]
		has only the integral solution $(c,d) = (0,0)$, a contradiction by Lemma \ref{lem:nonzero-homology-2-dim}.

		\item If $R(3)([F_2]) = [F_2]$ then $[F_2] \in \Z\{H, E_1+ E_2\}$. This means that $\psi([F_2]) \neq -[F_2]$. Compute that 
		\[
			\Z\{H, E_1+ E_2\} \cap \Z\{H-E_1-E_2+E_3, H-E_1, H - E_2\} = \Z\{2H - E_1-E_2\}.
		\]
		This shows that $[F_2] =d(2H - E_1 - E_2)$ for some $d \in \Z$, i.e. $Q_{M_3}([F_2], [F_2]) = 2d^2$. Finally, compute
		\[
			Q_{M_3}([F_1], [F_1]) + Q_{M_3}([F_2],[F_2]) = 2c^2 + 2d^2 = 0
		\]
		only has the integral solution $(c,d) = (0,0)$, a contradiction of Lemma \ref{lem:nonzero-homology-2-dim}.\qedhere
	\end{enumerate}
\end{enumerate}
\end{proof}

It remains to show one realizability result before turning to the proof of the main theorem of this section.
\begin{cor}[{\cite[Theorem 8.4.2]{dolgachev}}]\label{cor:w3}
The group $G= \langle \psi ,\sigma(12), \sigma(23), -I_4 \rangle$ is realized by biholomorphisms and anti-biholomorphisms of $X = \Bl_P\CP^2$, where $P = \{[1:0:0], [0:1:0], [0:0:1]\}$.
\end{cor}
\begin{proof}
By \cite[Theorem 8.4.2]{dolgachev}, the algebraic automorphism group is $\Aut(X) \cong (\C^\vee)^2 \rtimes (S_3 \times S_2)$, where the factor $S_3 \times S_2$ is a lift of $\langle \psi \rangle \times \langle \sigma(12), \sigma(23) \rangle$. More explicitly, \cite[p. 388]{dolgachev} gives the lift which we describe below.

Consider the maps $h_{\psi}, h_{\sigma(ij)}: \CP^2 \dashrightarrow \CP^2$ for $(ij) = (12), (23)$ given by
\begin{align*}
h_\psi: [X:Y:Z] &\mapsto [YZ: XZ:XY], \\
h_{\sigma(12)}: [X:Y:Z] &\mapsto [Y:X:Z], \\
h_{\sigma(23)}: [X:Y:Z] &\mapsto [X:Z:Y].
\end{align*}
Note that $P$ is the set the basepoints of the rational map $h_{\psi}$, and $h_{\psi}$ induces an automorphism $f_{\psi}$ of $X$. The maps $h_{\sigma(12)}$ and $h_{\sigma(23)}$ permute the points in $P$ and so induce automorphisms $f_{\sigma(12)}$ and $f_{\sigma(23)}$ respectively of $X$. Note that $h_\psi \circ h_{\sigma(ij)} = h_{\sigma(ij)} \circ h_\psi$ for each $(ij)$ and $\langle h_{\sigma(12)}, h_{\sigma(23)} \rangle \cong S_3$. Hence $\langle f_\psi, f_{\sigma(12)}, f_{\sigma(23)} \rangle \cong \Z/2\Z \times S_3$. Moreover, compute that $[f_\psi] = \RRef_{H - E_1 - E_2 - E_3}$ and $[f_{\sigma(ij)}] = \RRef_{E_i - E_j}$ for $(ij) = (12), (23)$. 

On the other hand, consider the map $h_{-I_4}: \CP^2 \to \CP^2$ given by $h_{-I_4}: [X:Y:Z] \mapsto [\overline X: \overline Y: \overline Z]$. Then $h_{-I_4}$ fixes each point in $P$ and so induces an order $2$ diffeomorphism $f_{-I_4}$ on $X$ with $[f_{-I_4}] = -I_4$. Moreover, $h_{-I_4}$ commutes with each $h_\psi$, $h_{\sigma(ij)}$. Therefore, $\langle f_g: g \in G \rangle \leq \Diff^+(M_3)$ is a lift of $G$ to $\Diff^+(M_3)$. 
\end{proof}

We are ready to prove the main theorem of this section, Theorem \ref{thm:m3}.
\begin{proof}[Proof of Theorem \ref{thm:m3}]
By Lemma \ref{lem:max-finite-subgroups}, the maximal finite subgroups of $\Mod(M_3)$ are conjugate in $\Mod(M_3)$ to
\begin{enumerate}
	\item $\langle G_{H-E_1-E_2-E_3}, -I_4 \rangle = \langle \RRef_{E_1-E_2}, \RRef_{E_2-E_3}, \RRef_{E_3}, -I_4 \rangle = \langle \sigma(12),\sigma(23), R(3), -I_4 \rangle$,
	\item $\langle G_{E_2-E_3}, -I_4 \rangle = \langle \RRef_{H-E_1-E_2-E_3}, \RRef_{E_1-E_2}, \RRef_{E_3}, -I_4 \rangle = \langle \psi, \sigma(12), R(3), -I_4 \rangle$, or
	\item $\langle G_{E_3}, -I_4 \rangle = \langle \RRef_{H-E_1-E_2-E_3}, \RRef_{E_1-E_2}, \RRef_{E_2-E_3}, -I_4 \rangle = \langle \psi, \sigma(12), \sigma(23), -I_4 \rangle$. 
\end{enumerate}
Propositions \ref{prop:non-realizable-1} and \ref{prop:non-realizable-2} show that $\langle G_{H-E_1-E_2-E_3}, -I_4 \rangle$ and $\langle G_{E_2-E_3}, -I_4 \rangle$ respectively do not lift to $\Diff^+(M_3)$. Corollary \ref{cor:w3} shows that $\langle G_{E_3}, -I_4 \rangle$ lifts to $\Aut(S) \times \langle \tau \rangle$ as described in the statement of the theorem. The map $\pi|_{\Diff^+(M_3)}: \Diff^+(M_3) \to \Mod(M_3)$ is surjective by Theorem \ref{thm:diffeo-realizable}. For any maximal finite subgroup $H = gGg^{-1}$ with $G$ one of the maximal finite subgroups listed above, let $\alpha \in \Diff^+(M_3)$ be a representative of $g \in G$. Then if a lift $\tilde G$ of $G$ to $\Diff^+(M_3)$ exists then $\tilde H = \alpha \tilde G \alpha^{-1}$ is a lift of $H$ and vice versa. Therefore, a lift $\tilde G$ of $G$ to $\Diff^+(M_3)$ exists if and only if $G$ is conjugate to $\langle G_{E_3}, -I_4 \rangle$.
\end{proof}

\subsection{Complex Nielsen realization problem for rational manifolds}\label{sec:complex-automorphisms}
In this section we show that the complex Nielsen realization problem differs from the smooth Nielsen realization problem for any $M = M_*$ or $M_n$ for all $n \geq 0$, which we call \emph{rational manifolds}. We also show that for all $n \geq 1$, there exist finite order mapping classes $c \in \Mod(M_n)$ such that there exists no lift of $c$ to a biholomorphic or anti-biholomorphic map in $\Diff^+(M_n)$ for any complex structure of $M_n$. This justifies why we consider complex equivariant connected sums rather than biholomorphisms and anti-biholomorphisms in solving the Nielsen realization problem for del Pezzo manifolds. 

A result of Friedman--Qin (\cite[Corollary 0.2]{friedman--qin}) says that if $X$ is a complex surface diffeomorphic to a rational surface then $X$ is a rational surface. Because any blowup of $\CP^2$ at finitely many points and $\CP^1 \times \CP^1$ are rational surfaces, any complex structure of $M = M_*$ or $M_n$ for $n \geq 0$ turns $M$ into a rational surface. By \cite[Theorem 3.4.6]{gompf--stipsicz}, the complex surface $M$ is geometrically ruled if $M = M_*$ or $M_n$ with $n \geq 1$ and biholomorphic to $\CP^2$ if $M = M_0$. Simply-connected, minimal, geometrically ruled surfaces are isomorphic to a Hirzebruch surface $\mathbb F_m$ for some $m \geq 0$ and $m \neq 1$ by (\cite[Theorem 3.4.8]{gompf--stipsicz}). This implies that the complex surface $M_n$ is a blowup of $\CP^2$ at $n$ points or a blowup of a Hirzebruch surface $\mathbb F_m$ at $(n-1)$ points.

The following lemma gives some examples of restrictions on the possible fixed sets of biholomorphisms and anti-biholomorphisms.
\begin{lem}\label{lem:complex-aut-obstruction}
The fixed set of a biholomorphism of finite order is orientable. The fixed set of an anti-biholomorphic involution has no isolated points.
\end{lem}
\begin{proof}
Let $\Phi \in \Diff^+(M)$ and consider an almost complex structure $J$ on $M$. Note that $\Fix(\Phi)$ is a disjoint union of surfaces and isolated points.
\begin{enumerate}
	\item If $\Phi\in \Diff^+(M)$ is biholomorphic then the linear operators $d\Phi_p$ and $J$ on $T_pM$ commute for all $p \in \Fix(\Phi)$. Because $T_p(\Fix(\Phi))$ is the fixed subspace of $T_pM$ under $d\Phi_p$, the space $T_p(\Fix(\Phi))$ is preserved by $J$. Therefore, $\Fix(\Phi)$ is orientable because $J|_{T\Fix(\Phi)}$ is an almost complex structure on $\Fix(\Phi)$. 

	\item If $\Phi\in \Diff^+(M)$ is anti-biholomorphic then $d\Phi_p \circ J = -J \circ d\Phi_p$ for all $p \in \Fix(\Phi)$, meaning that $d\Phi_p$ is not in the center of $\GL(T_pM)$. On the other hand, observe that if $\Phi$ is an involution with an isolated fixed point $p \in M$, the differential $d\Phi_p$ acts by negation on $T_pM$, but the negation map $-I_4$ is in the center of $\GL(T_pM)$. \qedhere
\end{enumerate}
\end{proof}

In the following proposition we apply Lemma \ref{lem:complex-aut-obstruction} and the results from Section \ref{sec:finite-gp} about fixed sets of smooth actions of finite groups on $4$-manifolds to derive some contradictions.
\begin{prop}\label{prop:bihol-antibihol}
Fix $n \geq 1$. Let $c \in \Mod(M_n)$ be
\[
	c = \begin{cases}
	\RRef_H & \text{ if }n = 1 \\
	\RRef_{E_1} \RRef_{E_2} & \text{ if }n = 2 \\
	\RRef_H \RRef_{E_1-E_2} & \text{ if }n = 3 \\
	\RRef_H\prod_{k=1}^{n-1} \RRef_{E_k} & \text{ if }n \geq 4.
	\end{cases}
\]
The mapping class $c \in \Mod(M_n)$ is not realizable by a biholomorphim or an anti-biholomorphim of order $2$ of any complex structure of $M_n$.
\end{prop}
\begin{proof}
Suppose $f \in \Diff^+(M_n)$ is of order $2$ with $[f] = c$. By Remark \ref{rmk:euler-char}, $\chi(\Fix(f)) = 3-n$ if $n \neq 3$ and $\chi(\Fix(f)) = 2$ if $n = 3$; therefore, $\Fix(f) \neq \emptyset$ for all $n \geq 1$. By Proposition \ref{prop:edmonds}, the fixed set $\Fix(f)$ must satisfy 
\[
	\beta_1(\Fix(f)) = n, \quad \beta_0(\Fix(f)) + \beta_2(\Fix(f)) = 3
\]
if $n \neq 3$ and $\beta_1(\Fix(f)) = 1$ and $\beta_0(\Fix(f)) + \beta_2(\Fix(f)) = 3$ if $n = 3$.
Because $\Fix(f)$ is a disjoint union of surfaces and isolated points, $\Fix(f) \cong S \sqcup \{p\}$, where $S$ is a connected surface and $p$ is an isolated fixed point of $f$. Lemma \ref{lem:complex-aut-obstruction} implies that $f$ cannot be anti-biholomorphic because it fixes an isolated point. We now prove that $f$ cannot be biholomorphic.

If $n \equiv 1 \pmod 2$ then $\beta_1(S) \equiv 1 \pmod 2$ and so $S$ is non-orientable. Suppose that $n\equiv 0 \pmod 2$ and that $S$ is orientable. The fixed subspace $H_2(M_n; \Z)$ by $c$ is $\Z\{E_n\}$ if $n \neq 2$ and $\Z\{H\}$ if $n = 2$. By the $G$-signature theorem (Theorem \ref{thm:g-signature}),
\[
	(1-n) + Q_{M_n}([S], [S]) = \begin{cases}
	-2 & \text{if }n \neq 2, \\
	2 & \text{if }n = 2.
	\end{cases}
\]
\begin{enumerate}
	\item If $n \neq 2$ then $Q_{M_n}([S], [S]) = n - 3$. On the other hand, $c([S]) = [S]$ and so $[S] = aE_n$ for some $a \in \Z$. This implies that $Q_{M_n}([S], [S]) = -a^2 = n-3$. However, $Q_{M_n}([S], [S]) = n-3 > 0$ because $n \geq 4$, but $-a^2\leq 0$. 
	\item If $n = 2$ then $Q_{M_2}([S], [S]) = 3$. On the other hand, $c([S]) = [S]$ and so $[S] = aH$ for some $a \in \Z$, and so $Q_{M_2}([S], [S]) = a^2 = 3$. There exists no $a \in \Z$ such that $a^2 = 3$. 
\end{enumerate}

In any case, $S$ must be non-orientable. Lemma \ref{lem:complex-aut-obstruction} implies that $f$ is not biholomorphic.
\end{proof}

The following lemma shows that the of the mapping classes $c \in \Mod(M_n)$ for $n \neq 2$ considered in Proposition \ref{prop:bihol-antibihol} are realizable by complex equivariant connected sums.
\begin{lem}\label{lem:n-geq-1-cecs}
If $n = 3$, let
\[
	c = \RRef_H \RRef_{E_1-E_2}\in \Mod(M_3).
\]
If $n \geq 4$, let 
\[
	c = \RRef_H \prod_{k=1}^{n-1}\RRef_{E_k}\in \Mod(M_n).
\]
Then $c$ is realizable by a complex equivariant connected sum. 
\end{lem}
\begin{proof}
If $n \neq 3$, consider the blowup $\Bl_P \CP^2$ where $P \subseteq \CP^2$ is a subset of $(n-1)$-points that are fixed by $\tau: \CP^2 \to \CP^2$ given by complex conjugation. Then $\tau$ induces an anti-biholomorphism $g: \Bl_P \CP^2 \to \Bl_P \CP^2$ of order $2$ which induces the negation map on $H_2(M_{n-1}; \Z)$ and fixes a surface $S \subseteq M_{n-1}$. If $n = 3$, consider $c_0 = \RRef_{H} \RRef_{E_1-E_2} \in \Mod(M_2)$. Theorem \ref{thm:m2} and Corollary \ref{cor:finite-order-m2} shows that $c_0$ is realized by a complex equivariant connected sum with a smooth $\Z/2\Z$-action given by a diffeomorphism $g$. By \cite[Corollary 1.4]{edmonds}, $\chi(\Fix(g)) = 1$ and by Proposition \ref{prop:edmonds}, $\beta_1(\Fix(g)) = 1$. Therefore, $\Fix(g) \cong S := \RP^2 \subseteq M_2$.

On the other hand, consider the diffeomorphism $f: \overline{\CP^2} \to \overline{\CP^2}$ defined by $[A:B:C] \mapsto [-A:B:C]$; it fixes a surface $\mathbb S^2 \subseteq \overline{\CP^2}$. 

For any $n \geq 3$, pick a point $p \in S \subseteq M_{n-1}$ fixed by $g$ and $q \in \mathbb S^2 \subseteq \overline{\CP^2}$. Both $dg_p$ and $df_q$ act by the linear maps $(a,b,c,d) \mapsto (-a, -b, c, d)$ with respect to appropriate positive bases $B_{M_n}$ and $B_{\overline{\CP^2}}$ of $T_pM_{n-1}$ and $T_q \overline{\CP^2}$ respectively. The tangential representations $\langle g \rangle \to \SO(T_p M_{n-1})$ and $\langle f \rangle \to \SO(T_q \overline{\CP^2})$ therefore are equivalent by an orientation-reversing isomorphism $T_pM_{n-1} \to T_q \overline{\CP^2}$ given by $(a, b, c, d) \mapsto (b, a, c, d)$ with respect to the bases $B_{M_n}$ and $B_{\overline{\CP^2}}$. Therefore, there exists a complex equivariant connected sum $M_{n-1} \# \overline{\CP^2}$ with a smooth $\Z/2\Z$-action given by a diffeomorphism $g \# f$ acting by $g$ on $M_{n-1} - \{p\} \subseteq M_{n-1} \# \overline{\CP^2}$ and acting by $f$ on $\overline{\CP^2} - \{q\} \subseteq M_{n-1} \# \overline{\CP^2}$. By construction, $g \# f$ acts by negation on the first factor $H_2(M_{n-1}; \Z)$ of the direct sum $H_2(M_n; \Z) \cong H_2(M_{n-1}; \Z) \oplus H_2(\overline{\CP^2}; \Z)$ if $n \geq 4$ and by $c_0$ if $n = 3$. Moreover, $g\# f$ acts by the identity on the second factor $H_2(\overline{\CP^2}; \Z)$. This is precisely the action of $c$ on $H_2(M_n; \Z)$.
\end{proof}

Finally, we turn to the complex Nielsen realization problem for all rational manifolds. 
\begin{proof}[Proof of Theorem \ref{thm:complex-nielsen}]
Consider $c \in \Mod(M_0)$ acting on $H_2(M_0; \Z)$ by negation. Let $f$ be any diffeomorphism of order $2$ with $[f] = c$. By Remark \ref{rmk:euler-char}, $\chi(\Fix(f)) = 1$ so by Proposition \ref{prop:edmonds}, the fixed set $\Fix(f)$ satisfies
\[
	\beta_1(\Fix(f)) = 1, \quad \beta_0(\Fix(f)) + \beta_2(\Fix(f)) = 2.
\]
Combining with the fact that $\Fix(f)$ is a disjoint union of surfaces and isolated points implies that $\Fix(f) \cong \RP^2$. By Lemma \ref{lem:complex-aut-obstruction}, $f$ is not biholomorphic for any complex structure of $M$.

Consider $c \in \Mod(M_*)$ acting on $H_2(M_*; \Z)$ by negation. Let $f$ be any diffeomorphism of order $2$ with $[f]= c$. By a result of Friedman--Qin (\cite[Corollary 0.2]{friedman--qin}), any complex structure on $M_*$ turns $M_*$ into a minimal rational surface. Therefore, $M_*$ is a Hirzebruch surface $\mathbb F_n$ for some $n \equiv 0 \pmod 2$ by \cite[Theorems 3.4.6, 3.4.8]{gompf--stipsicz}. Consider $\mathbb F_n$ as the $\CP^1$-bundle over $\CP^1$ with the projection $p: \mathbb F_n \to \CP^1$. Let $C \subseteq \mathbb F_n$ be the image of a section of $p$. Then $C$ is a complex submanifold of $\mathbb F_n$ isomorphic to $\CP^1$ and $f(C)$ is a smooth submanifold of $\mathbb F_n$ diffeomorphic to $\CP^1$. The restriction $(p \circ f)|_{C}: \CP^1 \to \CP^1$ has degree $-1$ because $(p\circ f)_*([C]) = p_*\circ f_*([C]) = -[\CP^1]$. Therefore $(p \circ f)|_{C}$ is not holomorphic because it has negative degree. Because $C$ is a complex submanifold of $\mathbb F_n$ and $p$ is holomorphic, $f$ is not biholomorphic.

For $M = M_0$ and $M_*$, Proposition \ref{prop:m-star-and-m1} shows that there is a section of the map $\pi: \Diff^+(M) \to \Mod(M)$ and $\Mod(M)$ is realized by a complex equivariant connected sum. Therefore, the mapping class $c \in\Mod(M)$ considered above is realizable by a complex equivariant connected sum.

It remains to consider the cases $M = M_n$ with $n \geq 1$. Let $c \in \Mod(M)$ be the mapping class given in the statement of Proposition \ref{prop:bihol-antibihol}. Then $c$ is not realizable by a biholomorphism or an anti-biholomorphism of order $2$ of any complex structure of $M$. For $M = M_1$, Proposition \ref{prop:m-star-and-m1} shows that there is a section of the map $\pi: \Diff^+(M) \to \Mod(M)$ and $\Mod(M)$ is realized by a complex equivariant connected sum. For $M = M_2$, Theorem \ref{thm:m2} and Corollary \ref{cor:finite-order-m2} show that $c$ is realized by a complex equivariant connected sum. For $M = M_n$ with $n \geq 3$, Lemma \ref{lem:n-geq-1-cecs} shows that $c$ is realized by a complex equivariant connected sum. 
\end{proof} 

\small
\bibliographystyle{alpha}
\bibliography{nielsen-high}

\bigskip
\noindent
Seraphina Eun Bi Lee \\
Department of Mathematics \\
University of Chicago \\
\href{mailto:seraphinalee@uchicago.edu}{seraphinalee@uchicago.edu}

\end{document}